\documentclass[a4paper,12pt]{article}  
\pdfoutput=1
\usepackage{natbib} 
\usepackage{amsthm}
\usepackage{tikz} 
\usepackage{times} 
\usepackage{microtype}
\usepackage{xcolor}
\usepackage{pict2e} 
\usepackage{amssymb}
\usepackage{graphicx}
\usepackage{enumerate}
\usepackage{amsmath}
\usepackage{algorithm2e}
\usepackage{algpseudocode}
\usepackage[shortlabels]{enumitem}
\usepackage[hyphens]{url} 
\usepackage[colorlinks,linkcolor=purple,citecolor=blue]{hyperref}
\usepackage{cleveref}
\usepackage{thm-restate} 
\usepackage[format=hang,justification=justified,labelfont=bf,labelsep=quad]{caption} 
\sfcode`E=1000  
\sfcode`P=1000  
\newdimen\einr
\newdimen\eeinr 
\def\aabs#1{\par\hangafter=1\hangindent=\eeinr
    \noindent\hbox to\eeinr{\strut\hskip\einr#1\hfill}\ignorespaces}
\def\rmitem#1{\par\hangafter=1\hangindent=\einr
  \noindent\hbox to\einr{\ignorespaces#1\hfill}\ignorespaces} 
\newcommand\bullitem{\rmitem{\raise.17ex\hbox{\kern7pt\scriptsize$\bullet$}}} 
\def\subbull{\vskip-.8\parskip\aabs{\raise.2ex\hbox{\footnotesize$\circ$}}}

\newtheorem{theorem}{Theorem}[section]
\newtheorem{corollary}[theorem]{Corollary}

\newtheorem{lemma}[theorem]{Lemma}
\newtheorem{proposition}[theorem]{Proposition}
\newtheorem{remark}[theorem]{Remark}
\newtheorem{definition}[theorem]{Definition}

\newcounter{claimNb}[theorem]
\newtheorem{claim}[claimNb]{Claim}
\newtheorem*{claim*}{Claim}

\newenvironment{cproof}
{\noindent{\em Proof of Claim.\enspace}}
{\hfill\strut\nobreak\hfill$\Diamond$\par\medbreak}

\def\tombstone{\hbox{\lower.4pt\vbox{\hrule\hbox{\vrule
  \kern7.6pt\vrule height7.6pt}\hrule}\kern.5pt}}

\def\zw#1\par{\vskip2ex{\textbf{#1}}\par\nobreak} 
\newdimen\pix  
\newsavebox{\figA} 
\title{%
Cosigning Crossing Families and Outer-Planar Gadgets
}

\author{
Ahmad Abdi \and Mahsa Dalirrooyfard \and Meike Neuwohner 
}

\date{\today
}

\begin{document}
\maketitle

\begin{abstract}
Let $F$ be a crossing family over ground set $V$, that is, for any two sets $U,W\in{F}$ with nonempty intersection and proper union, both sets $U\cap{W},U\cup{W}$ are in $F$. Let $\sigma:V\to \{+,-\}$ be a signing. We call $\sigma$ a \textit{cosigning} if every set includes a positive element and excludes a negative element. It is \textit{$\cap\cup$-closed} if every pairwise nonempty intersection and co-intersection include positive and negative elements, respectively.

We characterize the existence of ($\cap\cup$-closed) cosignings $\sigma$ through necessary and sufficient conditions. Our proofs are algorithmic and lead to elegant `forcing' algorithms for finding $\sigma$, reminiscent of the Cameron-Edmonds algorithm for bicoloring balanced hypergraphs. We prove that the algorithms run in polynomial time, and further, the cosigning algorithm can be run in oracle polynomial time through an application of submodular function minimization. 

Cosigned crossing families arise naturally in digraphs with vertex set $V$ comprised of sources and sinks, where every set in $F$ is \textit{covered} by an incoming arc. Under mild and necessary conditions, we build an outer-planar arc covering of $F$ when the vertices are placed around a circle. These gadgets are then used to find disjoint dijoins in $0,1$-weighted planar digraphs when the weight-$1$ arcs form a connected component that is not necessarily spanning.
\vspace{0.5cm}

\end{abstract}

\section{Introduction}
Let $V$ be a finite set, and let $F\subseteq{2^V}\backslash\{\emptyset, V\}$ be a \textit{crossing family} of subsets of~$V$, that is, $U\cap{W}, U\cup{W}\in{F}$ for any two sets $U,W\in F$ that \textit{cross}, i.e., $U\cap{W}\neq{\emptyset}$ and $U\cup{W}\ne{V}$. Crossing families form a basic class of set families amenable to `uncrossing', which has proven to be a powerful technique in graph orientations and combinatorial optimization (see~\cite{Schrijver03}, Chapters 48, 49, 55, 60, and \cite{Frank11}, Part III).

Motivated by a conjecture on graph orientations by Chudnovsky, Edwards, Kim, Scott, and Seymour~\cite{Chudnovsky16}, which was recently proved by the authors~\cite{ADN25}, 
we introduce new notions and problems on crossing families, which are of independent interest. We then show an application to the recently proved conjecture.

A \textit{cosigning} of a set family $F\subseteq 2^V\backslash \{\emptyset,V\}$ is a signing $\sigma:V\rightarrow\{+,-\}$ such that every set in $F$ includes a \emph{positive} element $u$ ($\sigma(u)=+$) and excludes a \emph{negative} element $v$ ($\sigma(v)=-$). This notion naturally arises in digraphs over vertex set $V$ where every vertex is a source or a sink, and every $U\in F$ receives an incoming arc. Our first result gives necessary and sufficient conditions for crossing families that admit a cosigning. 

\begin{theorem}\label{thm:weak-signing}
Let $V$ be a finite set, and let $F\subseteq 2^V\backslash \{\emptyset,V\}$ be a crossing family. Then $F$ admits a cosigning if and only if every set in $F$ includes an element $u$ and excludes an element $v$ such that $V\backslash\{u\}$ and $\{v\}$ do not belong to $F$. 
\end{theorem}

The necessity is fairly clear: given a cosigning of \textit{any} set family $F$, for every set $U$ in the family, the condition holds for any positive element $u$ inside $U$, and any negative element $v$ outside. The crux of the theorem is in establishing the sufficiency of the condition for crossing families, which is done in Section~\ref{proof-weak-signing}; we shall outline our approach shortly after defining a stronger variant of cosignings. 

Given a set family $F\subseteq 2^V\backslash \{\emptyset,V\}$, a \textit{$\cap\cup$-closed\,\footnote{read as cap-cup-closed or intersection-union-closed} cosigning} is a signing $\sigma:V\rightarrow\{+,-\}$ such that $X\cap Y$ contains a positive element for all $X,Y\in F$ that intersect, and $X\cup Y$ excludes a negative element for all $X,Y\in F$ that \emph{co-intersect}, i.e., $X\cup Y\neq V$. Observe that when $X=Y=U$ we recover the condition that defines a  cosigning, so every $\cap\cup$-closed cosigning is indeed a cosigning. 

We give the following necessary and sufficient conditions for crossing families that admit $\cap\cup$-closed cosignings. Necessity follows in the exact same manner as in \Cref{thm:weak-signing} (for all set families).


\begin{theorem}\label{thm:strong-signing}
Let $V$ be a finite set, and 
let $F\subseteq {2^V}\backslash\{\emptyset,V\}$ be a crossing family.
Then $F$ admits a $\cap\cup$-closed cosigning if and only if for all $Z,T\in{F}$ that intersect, $Z\cap{T}$ includes an element $u$ such that $V\backslash\{u\}\neq{B\cup{W}}$  for all $B,W\in{F}$ and for all $X,Y\in{F}$ that co-intersect, $X\cup{Y}$ excludes an element $v$ such that $\{v\}\neq{H\cap{G}}$ for all $G,H\in{F}$.  
\end{theorem}


The proofs of sufficiency for both Theorems~\ref{thm:weak-signing} and \ref{thm:strong-signing} are quite natural and algorithmic. 
\begin{quote} {\bf ($\cap\cup$-Closed) Cosigning Forcing Algorithm:} Given a crossing family satisfying the necessary conditions, if there is an unsigned element $u$ whose sign in a ($\cap\cup$-closed) cosigning is `forced' given the signed elements so far, then sign $u$ accordingly; otherwise, pick an unsigned element and sign it arbitrarily; repeat.
\end{quote}
The general descriptions of the algorithms above will be made more concrete in the proofs, which then amount to showing that at no iteration of the algorithms do we encounter a `conflict', a scenario in which an unsigned element is forced simultaneously to be both positive and negative. 

The Cosigning Forcing Algorithms are reminiscent of the Cameron-Edmonds algorithm for bicoloring the vertices of a balanced hypergraph such that no edge is monochromatic \cite{Edmonds-balance}. This was the first polynomial algorithm for bicoloring balanced matrices, and inspired the pursuit of the famous polynomial recognition algorithm for balanced hypergraphs nearly a decade later by Conforti, Cornu\'{e}jols, and Rao~\cite{Conforti99}. 

Letting $n:=|V|$ and $m:=|F|$, we prove that the Cosigning Forcing Algorithm has running time at most $m \cdot {n+1 \choose 2}$, while the $\cap\cup$-closed variant has running time at most $2\cdot {m\choose 2} \cdot {n+1 \choose 2}$. Many crossing families that arise in applications, however, have size exponential in $n$, so it might be more appropriate to switch to a different computational model in which $F$ is accessed via a suitable membership oracle. In this model, we prove that the Cosigning Forcing Algorithm can be executed in time at most $2 n^3 T$, where $T$ is the best running time of an oracle polynomial algorithm for computing the minimum of a modular function over a `well-provided lattice family' (see \Cref{sec:weak-oracle}). 
Finding a $\cap\cup$-closed cosigning in oracle polynomial time is left as an open problem.

Our next contribution is to a covering problem on planar digraphs involving $\cap\cup$-closed cosignings. 

\paragraph{The Circle Problem.} Let $V$ be a finite set of vertices, and let $\sigma$ be a $\cap\cup$-closed cosigning for a crossing family $F\subset 2^V\backslash \{\emptyset,V\}$ without complementary sets. Thus, in the digraph $D=(V,A)$ with an arc from every negative vertex to every positive vertex, every set in $F$ receives an incoming arc. Suppose now that the vertices in $V$ are placed at distinct positions around a circle, in a fixed order, drawn on the plane. Can we then move to an `outer-planar' subset of $A$ such that every set in $F$ still receives an incoming arc? Here, a set of arcs is \emph{outer-planar} if the arcs can be drawn inside the circle in a planar manner. While this is not always possible, see \Cref{sec:circle-examples}, we show that it can be guaranteed when every set in $F$ forms an interval around the circle. We summarize this in the following theorem.
\newpage
\begin{theorem}\label{thm:circle}
Let $V$ be a finite set of vertices drawn around a circle, and let $\sigma$ be a $\cap\cup$-closed cosigning for a crossing family $F\subseteq 2^V\backslash \{\emptyset,V\}$ without complementary sets. Suppose every set in $F$ forms an interval around the circle. Then there exists an outer-planar set of arcs from the negative to the positive vertices such that every set in $F$ receives an incoming arc.
\end{theorem}

None of the conditions on $F$ in \Cref{thm:circle} can be dropped while keeping the others; see Section~\ref{sec:circle-examples} for more. For example, consider four vertices $v_1,v_2,v_3,v_4$ arranged clockwise around the circle, with $\sigma(v_1)=\sigma(v_4)=-$ and $\sigma(v_2)=\sigma(v_3)=+$. If $F=\{\{v_1,v_2\},\{v_3,v_4\}\}$, no outer-planar arc set from negative to positive vertices can cover both sets of $F$. This set family satisfies all conditions of \Cref{thm:circle} except that it has complementary sets.\\
This theorem is useful in building planarity-preserving gadgets for digraphs. In particular, as our next contribution, we apply it to find disjoint dijoins in $0,1$-weighted planar digraphs. 

\paragraph{Disjoint Dijoins.} Let $D=(V,A)$ be a digraph. A \textit{dicut} is a cut of the form $\delta^+(U)\subseteq{A}$ where $U\subset V, U\neq \emptyset$ and $\delta^-(U)=\emptyset$. 
We call $U$ the \textit{shore} of the dicut $\delta^+(U)$. The dicut shores of $D$ form a crossing family of subsets of $V$. 
A blocking notion is that of a \textit{dijoin}, which is an arc subset that intersects every dicut at least once.

Consider arc weights $w\in \{0,1\}^A$ such that every dicut of $D$ contains at least two weight-$1$ arcs. Suppose further that the weight-$1$ arcs form one connected component, not necessarily spanning all the vertices. Chudnovsky et al.~\cite{Chudnovsky16} conjectured that the weight-$1$ arcs can then be decomposed into two dijoins of $D$. Recently, we proved this conjecture to be true by using integrality properties of the intersection of two submodular flow systems~\cite{ADN25}. In this paper, we use graph-theoretic techniques to prove the conjecture for planar digraphs.

\begin{theorem}\label{thm:main}
Let $(D,w)$ be a $0,1$-(arc-)weighted planar digraph where every dicut has weight at least $2$ and the weight-$1$ arcs form a connected component, not necessarily spanning. Then the weight-$1$ arcs contain two disjoint dijoins of $D$. 
\end{theorem}

Let us give a brief outline of the proof, in part motivating the circle problem. Chudnovsky et al. (\cite{Chudnovsky16}, 1.6 part 2) proved \Cref{thm:main} when the weight-$1$ arcs form a connected component that is \textit{spanning}. This seemingly harmless condition turns out to be a key assumption in their proof  which breaks down without it, and we do not see how to salvage it. Instead, we use their result in a black-box fashion. 

Suppose the weight-$1$ arcs form a connected component that is not spanning, and let $v$ be a \emph{weight-$0$ vertex}, i.e., a vertex incident only to weight-$0$ arcs. Our goal then is to delete $v$ and all the arcs incident with it, and to then install a planarity-preserving gadget connecting the in- and out-neighbors of $v$ that maintains some key features of the problem. We repeat this until the weight-$1$ arcs form a spanning connected component, and then apply the Chudnovsky et al.\ result to prove \Cref{thm:main}.

Toward this end, let $N^-(v),N^+(v)\subset{V}$ be the sets of in- and out-neighbors of $v$, respectively, placed around a circle in the same order as they appear around $v$ in the planar drawing of $D$. We then delete $v$ and all the arcs incident with it, and add a new arc from every vertex in $N^-(v)$ to every vertex in $N^+(v)$, as shown in \Cref{fig:weight-0-vertex-reduction}. Although this reduction essentially maintains the dicuts and their weights, as well as the dijoins, it can make the digraph non-planar. We therefore have to remove some of the newly added arcs, but in doing so we potentially create new dicuts and thus restrict the family of dijoins. This turns out be a nonissue if the minimum weight of a dicut remains at least $2$.

We carefully choose the following crossing family for which $N^-(v),N^+(v)$ yields a $\cap\cup$-closed cosigning:
$$\{U\subseteq{V\backslash{\{v\}}}\mid\delta^+(U)\text{ is a dicut of weight at most 1 in }{D\backslash\{v\}}\}$$ 
We prove in \Cref{sec:bridge-lemma-with-proof} that this family satisfies the hypotheses of \Cref{thm:circle}. We then apply this theorem to decide which of the newly added arcs to keep to regain planarity. 
\begin{figure}
    \centering
    \input{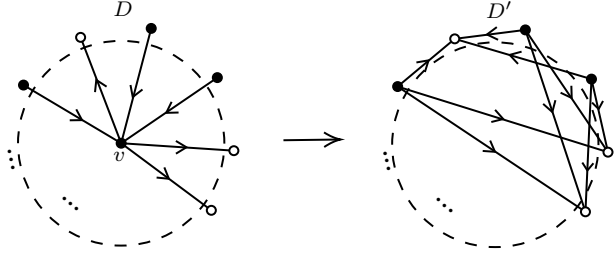}
    \caption[]{(Left) A weight-$0$ vertex $v$ and its in- and out-neighbors in a planar $0,1$-weighted digraph. (Right) A naive reduction that does not preserve planarity. \Cref{thm:circle} tells us how to retain an outer-planar subset of these arcs such that every dicut has weight at least $2$. Filled-in and non-filled-in vertices are the in and out neighbors of $v$, respectively. 
    }
    \label{fig:weight-0-vertex-reduction}
\end{figure}

\paragraph{Outline of The Paper.} 
\Cref{sec:weak-signing} studies cosignings of a crossing family, describes the Cosigning Forcing Algorithm in detail, and proves \Cref{thm:weak-signing}. After that, \Cref{sec:strong-singing} studies $\cap\cup$-closed cosignings and proves \Cref{thm:strong-signing}. \Cref{sec:circle-problem} explains in detail the necessity of the conditions of \Cref{thm:circle},  outlines its algorithmic proof and comments on its polynomial running time. After that, it formally states and proves the auxiliary lemmas, followed by the full proof of \Cref{thm:circle}. \Cref{sec:bridge-lemma-with-proof} establishes the lemma connecting Theorems~\ref{thm:circle} and ~\ref{thm:main}. Lastly, \Cref{sec:dijointhmproof} proves \Cref{thm:main}.

\section{Cosigning a Crossing Family}\label{sec:weak-signing}

In this section, we prove \Cref{thm:weak-signing}, providing necessary and sufficient conditions for when a crossing family admits a cosigning. As part of our proof, we devise a Cosigning Forcing Algorithm. We then analyze the algorithm's running time in two computational models depending on how the crossing family is accessed.
\subsection{The Cosigning Forcing Algorithm}\label{proof-weak-signing}

Let $V$ be a finite set, and let $F\subseteq 2^V\backslash \{\emptyset,V\}$ be a crossing family. We prove that $F$ admits a cosigning if and only if every set in $F$ includes an element $u$ and excludes an element $v$ such that $V\backslash\{u\}$ and $\{v\}$ do not belong to $F$. 

We argued necessity in the introduction already. 
To show sufficiency, assume that the conditions hold. We shall apply the Cosigning Forcing Algorithm from the introduction to sign the elements one by one; let us define `forcings' formally.

At any iteration of the algorithm, an unsigned element $v$ is \textit{forced to be positive} if there exists a set $U\in{F}$ such that $v\in{U}$ and all the elements in $U\backslash\{v\}$ are negatively signed; this forcing is \textit{trivial} if $\{v\}\in F$.

Similarly, an unsigned element $u$ is \textit{forced to be negative} if there exists a set $W\in{F}$ such that $u\notin{W}$ and all the elements outside $W$ other than $u$ are positively signed; this forcing is \textit{trivial} if $V\backslash \{u\}\in F$.

The algorithm starts by first signing all the trivially forced elements. We show that there is no element $w$ that is trivially forced to be both positive and negative, i.e., a \emph{conflict}, since otherwise $\{w\}, V\backslash\{w\}\in{F}$ thus contradicting our condition for either of the sets.

The algorithm then proceeds as follows. If there is an unsigned forced element, sign it accordingly --- we prove that this can be done without a conflict. If not, pick an unsigned element, and sign it arbitrarily. Repeat until all elements are signed.


Suppose for a contradiction that the algorithm encounters a conflict. Take the first one, which is an element $v$ that is forced to be both positive and negative. Then there exist sets $U,W\in{F}$ such that $v\in{U}$ with all the elements of $U\backslash\{v\}$ signed negatively, and $v\notin{W}$ with all the elements outside $W$ other than $v$ signed positively. In particular, all the elements in $U\cap W$ are negatively signed, and all the elements outside $U\cup W$ are positively signed. Observe that if one forcing is trivial, say $U=\{v\}$, then all signings so far have been trivially forced, in particular all the elements outside of $W$ are trivially forced to be positive, contradicting our assumption for this set. Thus, $v$ is not trivially forced, and so $|U|, |V\backslash W|\ge{2}$.
Let $u\in{U\backslash\{v\}}$, which is negatively signed, and therefore belongs to $W$.
Let $z\in{V\backslash(W\cup\{v\})}$, which is positively signed, and is therefore outside $U$.
Subsequently, $u\in U\cap W$ and $z\notin U\cup W$, so $U\cap W\neq \emptyset$ and $U\cup W\neq V$. Thus, as $F$ is a crossing family, we have $U\cap{W}\in{F}$. However, all the elements in $U\cap{W}$ are negatively signed, and this indicates an earlier conflict. 

More precisely, let $w$ be the last element that the algorithm signs from $U\cap{W}$, which occurs at an earlier iteration of the algorithm. Then $w$ must have been forced to be positive due to the remaining signs in $U\cap{W}\in F$. However, the algorithm decided to sign $w$ negatively, a decision that must have been forced. Therefore, $w$ was forced to both signs at an earlier iteration, thereby contradicting our choice of $v$.

We showed that the algorithm signs all the elements without a conflict. At the end, every set $U\in{F}$ must have a positive element inside. If not, the element $l$ that was signed last in $U\in F$ must have been forced to be positive, yet the algorithm forced it to be negative, indicating a conflict at $l$, a contradiction. Similarly, every set in $F$ must exclude a negatively signed element. Thus, the output is a cosigning, thereby finishing the proof.\qed

\subsection{Running Time Analysis}\label{sec:weak-oracle}

Given the setup of \Cref{thm:weak-signing},
the Cosigning Forcing Algorithm from the previous subsection takes time at most $m \cdot {n+1 \choose 2}$, where 
$n=|V|$ and $m=|F|$. In each iteration, we go through each set in $F$ at most once to identify a suitable unsigned forced element, thus taking time at most $m n'$ where $n'$ is the number of unsigned elements. Subsequently, the total running time of the algorithm is at most $\sum_{n'=1}^n mn' = m \cdot {n+1 \choose 2}$.



Let us analyze the running time in a different computational model where $V$ is given explicitly, but the crossing family $F$ is accessed via a \emph{membership oracle} which, given any $U\subseteq V$, tells us in unit time whether or not $U\in F$. 

To avoid exponential information-theoretic lower bounds, we also assume that the crossing family $F$ is \emph{well-provided}. This means that for all distinct $u,v\in V$, the subfamily $F_{uv}:=\{U\in F : u\in U, v\notin U\}$ is \emph{well-provided}. Briefly speaking,  $F_{uv}$ is a lattice family, and can therefore be provided compactly by indicating its unique inclusion-wise minimal and maximal sets along with a pre-order on the elements --- see \cite{Schrijver03}, \S 49.3 for more details.
To analyze the running time of the algorithm in this model, we need to describe a polynomial-time subroutine for deciding whether an unsigned element $v$ is forced to a particular sign. This can be done by signing $v$ the other way and then deciding whether the signing is \textit{invalid}, i.e., there is a set in $F$ whose elements inside are all negatively signed or whose elements outside are all positively signed.

Suppose we want to check whether an unsigned element $u$ is forced to be positive. We sign $u$ negatively and check whether there is a set $U\in F$ with $u\in U$ such that all elements in $U$ are negatively signed. Consider the vector $a\in \{0,1\}^V$ where $a_v=0$ if $v$ is negatively signed, and $a_v=1$ otherwise. Let $f:2^V\to \mathbb{Z}_{\geq 0}$ be the modular function defined as $f(U)=\sum_{v\in U} a_v$ for all $U\subseteq V$. Then $\min_{u\in U\in F} f(U) = 0$ if and only if there is a set in $U\in F$ with $u\in U$ whose elements inside are all negatively signed; in this case, any minimizer $U^\star$ will be such a set. To find it, we iterate over all $v\in V\setminus \{u\}$ and compute $\min_{U\in F_{uv}} f(U)$ using any algorithm for finding the minimum of a (sub)modular function over a lattice family, which can be done in oracle polynomial time  (see \cite{Schrijver03}, (49.25)). Since $\{V\backslash U:U\in F\}$ is a well-provided crossing family as well, we can check whether $u$ is forced to be negative analogously. If we denote the best running time of an oracle polynomial time algorithm for finding the minimum of a (sub)modular function over a lattice family by $T$, we can find a forced element (or decide that no element is forced) in time at most $2\cdot n\cdot (n-1)\cdot T$. Over all $n$ iterations, this yields a running time of at most $2n^3\cdot T$ in the oracle model.

\section{$\cap\cup$-Closed Cosignings of a Crossing Family}\label{sec:strong-singing}
In this section, we prove \Cref{thm:strong-signing}, providing necessary and sufficient conditions for when a crossing family admits a $\cap\cup$-closed cosigning. As part of our proof, we devise a $\cap\cup$-Closed Cosigning Forcing Algorithm. We then analyze the algorithm's running time when the crossing family is given explicitly.

\subsection{The $\cap\cup$-Closed Cosigning Forcing Algorithm}\label{proof-strong-signing}

Let $V$ be a finite set, and 
let $F\subset {2^V}\backslash\{\emptyset,V\}$ be a crossing family without complementary sets. 
We prove that $F$ admits a $\cap\cup$-closed cosigning if and only if for all $Z,T\in{F}$ that intersect, $Z\cap{T}$ includes an element $u$ such that $V\backslash\{u\}\neq{B\cup{W}}$  for all $B,W\in{F}$ and for all $X,Y\in{F}$ that co-intersect, $X\cup{Y}$ excludes an element $v$ such that $\{v\}\neq{H\cap{G}}$ for all $G,H\in{F}$.

Necessity of the conditions is straight-forward. To show sufficiency, assume that the conditions hold. We closely follow the start of the proof for cosignings.
We shall apply the $\cap\cup$-Closed Cosigning Forcing Algorithm from the introduction to sign the elements sequentially; let us define `forcings' rigorously.

At any iteration of the algorithm, an unsigned element $v$ is \textit{forced to be positive} if there exist $Z,T\in F$, possibly equal, such that $v\in Z\cap T$ and all the other elements in $Z\cap T$ are negatively signed; the forcing is \textit{trivial} if $Z\cap T=\{v\}$.

Similarly, an unsigned element $v$ is \textit{forced to be negative} if there exist $X,Y\in F$, possibly equal, such that $v\notin X\cup Y$ and all the other elements outside $X\cup Y$ are positively signed; the forcing is \textit{trivial} if $X\cup Y=V\backslash\{v\}$.

The algorithm kicks off by first signing all the trivially forced elements. We claim that there is no element $w$ that is trivially forced to be both positive and negative, i.e., a \emph{conflict}. If not, then $\{w\}=Z\cap T$ and $V\backslash\{w\}=X\cup Y$ for some $X,Y,Z,T\in{F}$, thus our assumption is violated for $Z$ and $T$, as well as for $X$ and $Y$. 

The algorithm then proceeds as follows. If there is an unsigned forced element, sign it accordingly --- we prove that this can be done without a conflict. If not, pick an unsigned element, and sign it arbitrarily. Repeat until all elements are signed.

Suppose for a contradiction that the algorithm encounters a conflict. Take the first one, which is an element $v$ that is forced to be both positive and negative.

Then there exist sets $X_1,X_2,Y_1,Y_2\in{F}$ such that $v\in X_1\cap X_2=:U$ and all the other elements in the intersection are signed negatively, and $v\notin Y_1\cup Y_2=:W$ and all the other elements outside the union are signed positively. Observe that if one forcing is trivial, say $X_1\cap{X_2}=\{v\}$, then all signings so far have been trivially forced. In particular all elements outside of $Y_1\cup{Y_2}$ (including $v$) are trivially forced to be positive, a contradiction to our assumption. Therefore, $v$ is not trivially forced to any sign.  
There are four cases:

\textbf{Case 1:} $X_1\cup X_2\neq V$ and $Y_1\cap Y_2\neq \emptyset$. Then $U,W$ both belong to the crossing family $F$. The proof of \Cref{thm:weak-signing} applies directly here to obtain a contradiction.

\textbf{Case 2:} $X_1\cup X_2\neq V$ and $Y_1\cap Y_2=\emptyset$. 
Then we have $U\in{F}$, $v\in{U}$ and all elements of $U\backslash\{v\}$ are signed negatively. 
Since $v$ is not trivially forced, let $u\in {U\backslash\{v\}}$; $u$ must be negatively signed. 
Since no negative element exists in $V\backslash(Y_1\cup{Y_2})$, we have $u\in{Y_1\cup{Y_2}}$. We may assume that $u\in{Y_1}$, so $U\cap{Y_1}\neq{\emptyset}$. Since $v$ is not trivially forced, let $z\in{V\backslash(Y_1\cup{Y_2}\cup\{v\})}$; $z$ is positively signed. 
Since $U$ has no positive element, $z\notin{U}$, so $U\cup{Y_1}\neq{V}$. Then since $F$ is a crossing family, we have $U\cap{Y_1}\in{F}$. 
Observe that all the elements in $U\cap{Y_1}$ are negative. 
Let $w$ be the last element in $U\cap{Y_1}$ signed negatively, earlier in the process. Then $w$ must have been forced to be positive by $U\cap{Y_1}$. However, $w$ received a negative sign in the algorithm, a decision that must have been forced. Thus, $w$ was forced to both signs earlier, contradicting that $v$ was the first conflict. 

\textbf{Case 3:} $X_1\cup X_2=V$ and $Y_1\cap Y_2\neq \emptyset$. Then we have $W\in{F}$, $v\notin{W}$, and all elements in $V\backslash(W\cup\{v\})$ are signed positively. 
Since $v$ is not trivially forced, let $x\in{V\backslash(W\cup\{v\})}$; $x$ is positively signed. 
Since no positively signed element exists in $U=X_1\cap{X_2}$, we have $x\notin{X_1\cap{X_2}}$, thus we may assume $x\notin{X_1}$, and so $W\cup{X_1}\neq{V}$. 
Since $v$ is not trivially forced, let $y\in{(X_1\cap{X_2})\backslash\{v\}}$; $y$ is signed negatively. Since no negative element is in $V\backslash{W}$, we have $y\in{W}$, and so $W\cap{X_1}\neq{\emptyset}$. Since $F$ is a crossing family, we have $W\cup{X_1}\in{F}$. Observe that all the elements of $V\backslash(W\cup{X_1})$ are positive. 
Let $z$ be the last element of $V\backslash(W\cup{X_1})$ signed positively, earlier in the process. Then $z$ must have been forced to be negative by $W\cup{X_1}$. However, $z$ was signed positively in the algorithm, a decision that must have been forced. Thus $z$ was forced to both signs earlier, contradicting that $v$ was the first conflict. 

\textbf{Case 4:} $X_1\cup X_2=V$ and $Y_1\cap Y_2=\emptyset$. 
Since $v$ is not trivially forced, let $x\in{(X_1\cap{X_2})\backslash\{v\}}$, which is signed negatively, and let $y\in{V\backslash(Y_1\cup{Y_2}\cup\{v\})}$, which is signed positively. Since no negative element exists in $V\backslash(Y_1\cup{Y_2})$, we have $x\in{Y_1\cup{Y_2}}$, and we may assume $x\in{Y_1}$. Moreover, since no positive element exists in $X_1\cap{X_2}$, we have $y\notin{X_1\cap{X_2}}$ and so we can assume $y\notin{X_1}$. As $x\in{X_1\cap{Y_1}}$ and $y\notin{X_1,Y_1}$, we have $X_1\cap{Y_1}\neq{\emptyset}$ and $X_1\cup{Y_1}\neq{V}$. Then since $F$ is a crossing family, we have $X_1\cup{Y_1}\in{F}$. 
Since $y\notin{X_1,Y_1,Y_2}$, we have  $(X_1\cup{Y_1})\cup{Y_2}\neq{V}$. Observe that 
all the elements in $V\backslash(X_1\cup{Y_1}\cup{Y_2})$ are positive. Let $t$ be the last element from $V\backslash(X_1\cup{Y_1}\cup{Y_2})$ that was signed. Then $t$ must have been forced to be negative by $X_1\cup{Y_1}$ and $Y_2$. However, $t$ received a positive sign in the algorithm, a decision that must have been forced. Thus $t$ was forced to both signs earlier, contradicting that $v$ was the first conflict. 

We showed that the algorithm signs all the elements without a conflict. At the end, $X\cap Y$ must contain a positive element for all $X,Y\in F$ that intersect. If not, the element $l$ that was signed last from $X\cap Y$ must have been forced to be positive, yet the algorithm forced it to be negative, indicating a conflict in the algorithm, a contradiction. Similarly, $Z\cup T$ must exclude a negative element for all $Z,T\in F$ that co-intersect. Thus, the output is a $\cap\cup$-closed cosigning, thus finishing the proof.\qed

\subsection{Running Time Analysis}

Given the setup of \Cref{thm:strong-signing},
the $\cap\cup$-Closed Cosigning Forcing Algorithm from the previous subsection takes time at most $2\cdot {m \choose 2} \cdot {n+1 \choose 2}$, where 
$n=|V|$ and $m=|F|$. In each iteration, we go through each pair of sets in $F$ twice to identify a suitable unsigned forced element, thus taking time at most $2\cdot {m \choose 2} \cdot n'$ where $n'$ is the number of unsigned elements. Subsequently, the total running time of the algorithm is at most $2\cdot {m \choose 2} \cdot {n+1 \choose 2}$.

\section{The Circle Problem}\label{sec:circle-problem}
In this section, we define the circle problem, which, roughly speaking, is concerned with the problem of covering a crossing family of circular intervals by an outer-planar set of arcs. More precisely, given a set of vertices $V$ around a circle $C$, a crossing family $F$ of circular intervals without complementary sets, and a $\cap\cup$-closed cosigning $\sigma$, our goal will be to find an outer-planar set of arcs, oriented from the set $V^-$ of negative vertices to the set $V^+$ of positive vertices, such that each interval in $F$ receives an entering arc.
We will do that algorithmically thus proving \Cref{thm:circle}. 
Our algorithm proceeds as follows: Starting with an empty arc set, we iteratively add a suitable outer-planar set of arcs and delete all sets from the family $F$ that receive an incoming arc, and repeat. This approach is motivated by the following lemma (proven in ~\Cref{sec:circle-lemmas}):
\begin{restatable}{lemma}{LemmaDeleteCoveredSets}\label{lem:covering-arcs} Let $(V,\sigma,F)$ be an instance of the circle problem. Let $E$ be a set of arcs from $V^-$ to $V^+$. Let $F'$ consist of all the sets in $F$ that are not covered by $E$. Then $(V,\sigma, F')$ is a valid instance of the circle problem.
\end{restatable}

The main obstacle in pursuing such a recursive strategy is how to choose the arc set $E$ in such a way that the new instance $(V,\sigma, F')$ admits a solution $E'$ that is compatible with $E$ in the sense that $E\cup E'$ remains outer-planar (and how to enforce this property without having to handle extra constraints in later iterations). To this end, we 
first observe that we can add arcs between adjacent vertices `for free'; these will never cross any other arcs. Adding such arcs and reducing the instance accordingly will be the first step of our algorithm. In general, we will of course not get away with only adding arcs between adjacent vertices and when adding an arc between non-adjacent vertices, this will separate the vertex sets into two parts, which can no longer be joined by any arcs. In the following, we will design a set of reduction rules that will, roughly speaking, add a set of arcs involving a small number of adjacent \textit{sign-blocks} (see Section~\ref{sec:preliminiaries-circle}), delete all covered sets, and then restrict the instance by deleting all endpoints of the added arcs except for the two outer ones. This will ensure that further arcs that we add do not cross any previously added one. The difficulty is to ensure that the deletion of elements produces a valid instance.

\subsection{Preliminaries and Notation\label{sec:preliminiaries-circle}}
For a set of $n$ vertices $V$ around a circle $C$ where each vertex is associated with a sign given by $\sigma:V\rightarrow\{+,-\}$, and a family $F$ of subsets of $V$, we say $(V,\sigma,F)$ is a (valid) instance of the circle problem if $F$ has the  following properties: 
\begin{enumerate}[label={\bfseries P\arabic*}, start=0]
    \item\label{P0} For every $U\in{F}$, $U$ is an interval of consecutive vertices around $C$.
    \item\label{P1} $F$ is a crossing family. 
    \item\label{P2} For every $U\in{F}$, there exists a positive vertex in $U$ and there exists a negative vertex in $V\backslash{U}$. 
    \item\label{P3} For every $U,W\in{F}$ such that $U\cap{W}=\emptyset$, there exists a negative vertex in $V\backslash(U\cup{W})$. 
    \item\label{P4} For every $U,W\in{F}$ such that $U\cup{W}={V}$, there exists a positive vertex in $U\cap{W}$.
\end{enumerate} 
Observe that properties~\ref{P2}-\ref{P4} are equivalent to stating that $\sigma$ is a $\cap\cup$-closed cosigning and $F$ does not contain a pair of complementary sets. ~\Cref{sec:circle-examples} provides an example and explains why each property is necessary for \Cref{thm:circle}. 

Let $(V,\sigma,F)$ be an instance of the circle problem. Denote by $V^+$ and $V^-$ the set of vertices of $V$ with $+$ and $-$ signs, respectively. 

For any set $U\subseteq{V}$, let $U^c=V\backslash{U}$. Denote by $U^+$ the set of $+$signed vertices in $U$ and by $U^-$ the set of $-$ signed vertices in $U$. Furthermore, let $U^{c+}$ be the set of $+$ signed vertices in $U^c$ and let $U^{c-}$ be the set of $-$ signed vertices in $U^c$. 

For any $V'\subseteq{V}$ and $U\subseteq V$, let $U[V']=U\cap{V'}$. We define the restriction of $F$ to the ground set $V'$ to be $F[V']=\{U[V']\mid U\in{F}, U[V']\ne \emptyset\}$. Furthermore, denote by $\sigma[V']$ the restriction of the sign function $\sigma$ to $V'$. 

Let $(m,p)$ be an arc from a $-$ signed vertex $m$ to a $+$ signed vertex $p$. We say $(m,p)$ \textit{covers} $U\in{F}$ if $m\notin{U}$ and $p\in{U}$. For a set of arcs $A$ from $V^-$ to $V^+$, let $F_A$ consist of all the sets $W\in{F}$ such that $W$ is covered by some arc in $A$. We say a set $A$ of arcs covers $F$ if $F_A=F$. 

For a set $U\in{F}$, a vertex $u\in{U}$ is an \textit{end vertex} if it is at the end of the interval $U$ forms on the circle. If $|U|=1$ then it has one end vertex and if $|U|\ge 2$, then it has two end vertices. If $v$ is an end vertex of $U$, an \textit{adjacent out-neighbor} of $v$ with respect to $U$ is a vertex $v'$ next to $v$ on the circle with $v'\notin{U}$. Note that by \ref{P2}, for any $U\in{F}$, we have $U\neq{V}$ and hence its end vertices have adjacent out-neighbors. 

We call a set of consecutive vertices with the same sign around the circle a \textit{block}. We say a set in $U\in{F}$ is a \textit{$1$-block} if it forms one block around the circle, i.e., all of its vertices have the same sign. Note that by \ref{P2}, the common sign is a $+$. We call a set $W\in{F}$ a \textit{$2$-block} if $W^+$ and $W^-$ are two nonempty blocks.  Moreover, we say a set $X\in{F}$ is a \textit{co-$2$-block} if $X^{c+}$ and $X^{c-}$ are two consecutive nonempty blocks. For a vertex $v$, the \textit{sign-block} of $v$ is the maximal block of vertices around the circle that contains $v$. 

We call a set of directed arcs $A$ from $V^-$ to $V^+$ \textit{outer-planar} if we can draw the arcs in $A$ inside the circle $C$ with no two arcs crossing each other. For an instance of the circle problem $(V,\sigma,F)$, we say that a set of directed arcs $E$ from $V^-$ to $V^+$ is a \textit{solution} to $(V,\sigma,F)$ if $E$ is outer-planar and covers $F$. 

For a given instance of the circle problem $(V,\sigma,F)$, we define its \textit{dual} $(V,\bar{\sigma},\bar{F})$ as follows: Let $\bar{\sigma}$ be the flipped signing, i.e., $\bar{\sigma}(v)=+$ if and only if $\sigma(v)=-$ and vice versa. Furthermore, let $\bar{F}=\{V\backslash{U}\mid U\in{F}\}$. 

\begin{remark}\label{rem:dual}
For a given instance of the circle problem $(V,\sigma,F)$, its dual $(V,\bar{\sigma}, \bar{F})$ is an instance of the circle problem as well, i.e., $\bar{F}$ has properties \ref{P0}-\ref{P4}. Furthermore, if $E$ is a solution to $(V,\sigma,F)$, then the arc set $\bar{E}$ that we obtain from $E$ by reversing the direction of each arc, is a solution to $(V,\bar{\sigma}, \bar{F})$. 
\end{remark}

\subsection{Necessity Examples}\label{sec:circle-examples}
In this section, we start by an example for \Cref{thm:circle}. After that we explain why the conditions~\ref{P0}-\ref{P4} are necessary and cannot be dropped. 

Let $V=\{v_1,v_2,...,v_9\}$ around a circle $C$ where the vertices $v_3,v_6,v_9$ have a negative sign and the rest of the vertices have a positive sign, as shown on the left in figure~\ref{fig:main-example}. Let $U_1=\{v_1,v_9\}, U_2=\{v_1, v_2, v_3\}, U_3=\{v_4, v_5, v_6\}, U_4 = \{v_1, v_4, v_5, v_6, v_7, v_8, v_9\}$ and let $F=\{U_1, U_2, U_3, U_4\}$. The sets of $F$ are shown in the middle in figure~\ref{fig:main-example}. 
It can readily checked that properties \ref{P0}-\ref{P4} hold for the vertices $V$ with the declared signing and the set $F$. Consider these set of arcs $\{(v_6,v_1), (v_3,v_5)\}$. This set is outer-planar, the arc $(v_6,v_1)$ is incoming to $U_1,U_2$ and the arc $(v_3,v_5)$ is incoming to $U_3, U_4$ as shown on the right in figure~\ref{fig:main-example}. 
\begin{figure}[hbt!]
    \centering
    \input{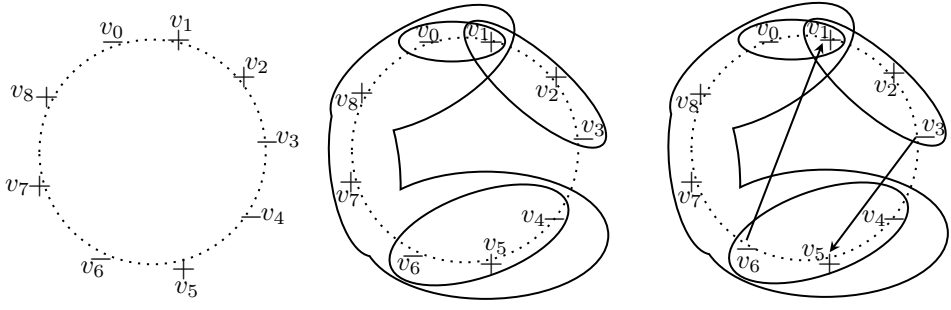}
    \caption[main-example]{An example for \Cref{thm:circle} with the outer-planar arc subset from $V^-$ to $V^+$}
    \label{fig:main-example}
\end{figure}

The reader might be curious about where the conditions of \Cref{thm:circle} come from and whether they are necessary for the outcome of the theorem. Condition~\ref{P2} is clearly necessary for the outcome. For the other conditions, we show that they cannot be dropped by providing multiple examples that have all the properties except one, and we show that the outcome of the theorem does not hold. For condition number $i\in\{0,1,3,4\}$, we use the notations $V^{(i)}$ as the vertex set, $\sigma^{(i)}$ as the sign function and $F^{(i)}$ as the family of subsets of the vertices.  

To show property~\ref{P0} cannot be dropped, let $V^{(0)}=\{v_1,v_2, v_3,v_4, v_5,v_6\}$ around a circle $C$ where $v_1,v_2,v_3$ have a positive sign and $v_4,v_5,v_6$ have a negative sign. Let $F^{(0)}=\{U_1,U_2,U_3\}$ where $U_1=\{v_1,v_6\}$, $U_2=\{v_2,v_5\}$ and $U_3=\{v_3,v_4\}$. The sets of $F$ are shown in figure~\ref{fig:drop-p0-example}. It is simple to check that the properties \ref{P1}-\ref{P4} hold for this example and \ref{P0} does not hold due to the set $U_2$. Now we argue that there is no outer-planar set of arcs from $V^{(0)-}$ to $V^{(0)+}$ for this example such that every set $U_i$ for $i\in\{1,2,3\}$ receives an arc. We try to construct the arc-set. 
For set $U_3$ to receive an arc, the arc should be from $v_5$ or $v_6$ to $v_3$. For set $U_1$ to receive an arc, the arc should be from $v_4$ or $v_5$ to $v_1$. If in the former, the arc comes from $v_6$, then $U_1$ cannot receive any arc while keeping the arc-set outer-planar. Hence, $U_3$ must receive an arc from $v_5$, forcing $U_1$ to also receive its arc from $v_5$. However, now the set $U_2$ cannot receive any arc to its only positive signed vertex $v_2$ while keeping the arc-set outer-planar. Therefore, the outcome of the theorem does not hold for this example. Observe that we can extend this example for any number of vertices by replacing vertex $v_2$ by any positive number of $+$ signed vertices, all belong only to $U_2$, and/or replacing vertex $v_5$ with any positive number of $-$ signed vertices, all belong only to $U_2$.   
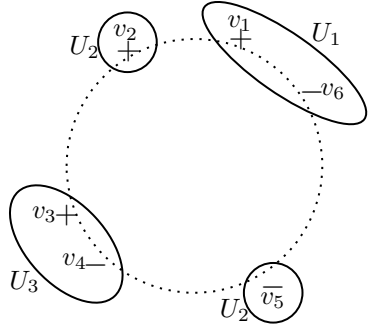
\begin{figure}
\centering
\tikzset{every picture/.style={line width=0.75pt}} 

\begin{tikzpicture}[x=0.75pt,y=0.75pt,yscale=-0.75,xscale=0.75]

\draw  [dash pattern={on 0.84pt off 2.51pt}] (62,175.91) .. controls (62,128.92) and (100.1,90.82) .. (147.09,90.82) .. controls (194.08,90.82) and (232.18,128.92) .. (232.18,175.91) .. controls (232.18,222.9) and (194.08,261) .. (147.09,261) .. controls (100.1,261) and (62,222.9) .. (62,175.91) -- cycle ;
\draw  [color={rgb, 255:red, 0; green, 0; blue, 0 }  ,draw opacity=1 ] (153.73,69.38) .. controls (160.1,60.35) and (189.09,69.87) .. (218.49,90.63) .. controls (247.89,111.39) and (266.56,135.54) .. (260.18,144.56) .. controls (253.81,153.58) and (224.82,144.07) .. (195.42,123.31) .. controls (166.02,102.54) and (147.35,78.4) .. (153.73,69.38) -- cycle ;
\draw  [color={rgb, 255:red, 0; green, 0; blue, 0 }  ,draw opacity=1 ] (83,93) .. controls (83,81.95) and (91.73,73) .. (102.5,73) .. controls (113.27,73) and (122,81.95) .. (122,93) .. controls (122,104.05) and (113.27,113) .. (102.5,113) .. controls (91.73,113) and (83,104.05) .. (83,93) -- cycle ;
\draw  [color={rgb, 255:red, 0; green, 0; blue, 0 }  ,draw opacity=1 ] (180,261) .. controls (180,249.95) and (188.73,241) .. (199.5,241) .. controls (210.27,241) and (219,249.95) .. (219,261) .. controls (219,272.05) and (210.27,281) .. (199.5,281) .. controls (188.73,281) and (180,272.05) .. (180,261) -- cycle ;
\draw  [color={rgb, 255:red, 0; green, 0; blue, 0 }  ,draw opacity=1 ] (29.8,193.18) .. controls (40.56,183.37) and (63.61,191.14) .. (81.28,210.53) .. controls (98.95,229.92) and (104.55,253.59) .. (93.78,263.4) .. controls (83.02,273.21) and (59.97,265.44) .. (42.3,246.05) .. controls (24.63,226.66) and (19.03,202.99) .. (29.8,193.18) -- cycle ;

\draw (168,82) node [anchor=north west][inner sep=0.75pt]   [align=left] {$\displaystyle +$}; 
\draw (215,117) node [anchor=north west][inner sep=0.75pt]   [align=left] {$\displaystyle -$};
\draw (189,248) node [anchor=north west][inner sep=0.75pt]   [align=left] {$\displaystyle -$};
\draw (52,200) node [anchor=north west][inner sep=0.75pt]   [align=left] {$\displaystyle +$};
\draw (71,233) node [anchor=north west][inner sep=0.75pt]   [align=left] {$\displaystyle -$};
\draw (93,90) node [anchor=north west][inner sep=0.75pt]   [align=left] {$\displaystyle +$};
\draw (167,73) node [anchor=north west][inner sep=0.75pt]  [font=\footnotesize] [align=left] {$\displaystyle v_{1}$};
\draw (91,79) node [anchor=north west][inner sep=0.75pt]  [font=\footnotesize] [align=left] {$\displaystyle v_{2}$};
\draw (37,201) node [anchor=north west][inner sep=0.75pt]  [font=\footnotesize] [align=left] {$\displaystyle v_{3}$};
\draw (57,233) node [anchor=north west][inner sep=0.75pt]  [font=\footnotesize] [align=left] {$\displaystyle v_{4}$};
\draw (189,259) node [anchor=north west][inner sep=0.75pt]  [font=\footnotesize] [align=left] {$\displaystyle v_{5}$};
\draw (229,119) node [anchor=north west][inner sep=0.75pt]  [font=\footnotesize] [align=left] {$\displaystyle v_{6}$};
\draw (22,244) node [anchor=north west][inner sep=0.75pt]  [font=\footnotesize] [align=left] {$\displaystyle U_{3}$};
\draw (223,78) node [anchor=north west][inner sep=0.75pt]  [font=\footnotesize] [align=left] {$\displaystyle U_{1}$};
\draw (62,87) node [anchor=north west][inner sep=0.75pt]  [font=\footnotesize] [align=left] {$\displaystyle U_{2}$};
\draw (162,263) node [anchor=north west][inner sep=0.75pt]  [font=\footnotesize] [align=left] {$\displaystyle U_{2}$};

\end{tikzpicture}
    \caption[example-drop-P0]{An example to show that \ref{P0} cannot be dropped from the conditions of  \Cref{thm:circle}}
    \label{fig:drop-p0-example}
\end{figure}

To show property~\ref{P1} cannot be dropped, let $V^{(1)}=\{v_1,v_2,v_3,v_4,v_5,v_6\}$ around a circle $C$ where $v_1,v_3,v_5$ have a negative sign and $v_2, v_4, v_6$ have a positive sign. Let $F^{(1)}=\{U_1,U_2\}$ where $U_1=\{v_1,v_2,v_3\}$ and $U_2=\{v_1,v_5,v_6\}$, as shown in figure~\ref{fig:drop-p1-example}. It can be readily checked that properties \ref{P0} and \ref{P2}-\ref{P4} hold but $F^{(1)}$ is not a crossing family. The only set of arcs from $V^{(1)-}$ to $V^{(1)+}$ such that both $U_1,U_2$ have an incoming arc, is the set $\{(v_5, v_2), (v_3,v_6)\}$, which is not outer-planar. Observe that we can extend this example for any number of vertices, by replacing vertex $v_1$ by any positive number of $-$ signed vertices, all belong to $U_1\cap{U_2}$, and/or replacing vertex $v_4$ by any positive number of $+$ signed vertices, all belong to $V\backslash(U_1\cup{U_2})$. 
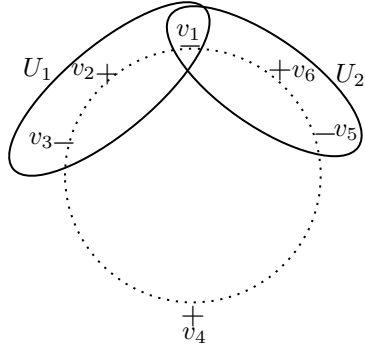
\begin{figure}
\centering
\tikzset{every picture/.style={line width=0.75pt}} 

\begin{tikzpicture}[x=0.75pt,y=0.75pt,yscale=-0.75,xscale=0.75]

\draw  [dash pattern={on 0.84pt off 2.51pt}] (401,151.91) .. controls (401,104.92) and (439.1,66.82) .. (486.09,66.82) .. controls (533.08,66.82) and (571.18,104.92) .. (571.18,151.91) .. controls (571.18,198.9) and (533.08,237) .. (486.09,237) .. controls (439.1,237) and (401,198.9) .. (401,151.91) -- cycle ;
\draw  [color={rgb, 255:red, 0; green, 0; blue, 0 }  ,draw opacity=1 ] (471.06,44.49) .. controls (480.44,31.2) and (516.06,40.21) .. (550.6,64.6) .. controls (585.15,89) and (605.55,119.55) .. (596.16,132.84) .. controls (586.78,146.13) and (551.16,137.13) .. (516.62,112.73) .. controls (482.07,88.33) and (461.67,57.78) .. (471.06,44.49) -- cycle ;
\draw  [color={rgb, 255:red, 0; green, 0; blue, 0 }  ,draw opacity=1 ] (494.48,40.19) .. controls (504.9,52.6) and (484.71,86.69) .. (449.39,116.34) .. controls (414.07,145.98) and (376.99,159.96) .. (366.58,147.55) .. controls (356.17,135.15) and (376.36,101.06) .. (411.68,71.41) .. controls (447,41.77) and (484.07,27.79) .. (494.48,40.19) -- cycle ;

\draw (534,72) node [anchor=north west][inner sep=0.75pt]   [align=left] {$\displaystyle +$};
\draw (564,115) node [anchor=north west][inner sep=0.75pt]   [align=left] {$\displaystyle -$};
\draw (474,56) node [anchor=north west][inner sep=0.75pt]   [align=left] {$\displaystyle -$};
\draw (476,237) node [anchor=north west][inner sep=0.75pt]   [align=left] {$\displaystyle +$};
\draw (390,121) node [anchor=north west][inner sep=0.75pt]   [align=left] {$\displaystyle -$};
\draw (419,75) node [anchor=north west][inner sep=0.75pt]   [align=left] {$\displaystyle +$};
\draw (473,49) node [anchor=north west][inner sep=0.75pt]  [font=\footnotesize] [align=left] {$\displaystyle v_{1}$};
\draw (403,75) node [anchor=north west][inner sep=0.75pt]  [font=\footnotesize] [align=left] {$\displaystyle v_{2}$};
\draw (375,121) node [anchor=north west][inner sep=0.75pt]  [font=\footnotesize] [align=left] {$\displaystyle v_{3}$};
\draw (477,252) node [anchor=north west][inner sep=0.75pt]  [font=\footnotesize] [align=left] {$\displaystyle v_{4}$};
\draw (578,116) node [anchor=north west][inner sep=0.75pt]  [font=\footnotesize] [align=left] {$\displaystyle v_{5}$};
\draw (550,76) node [anchor=north west][inner sep=0.75pt]  [font=\footnotesize] [align=left] {$\displaystyle v_{6}$};
\draw (579,75) node [anchor=north west][inner sep=0.75pt]  [font=\footnotesize] [align=left] {$\displaystyle U_{2}$};
\draw (371,71) node [anchor=north west][inner sep=0.75pt]  [font=\footnotesize] [align=left] {$\displaystyle U_{1}$};

\end{tikzpicture}
    \caption[example-drop-P1]{An example to show that \ref{P1} cannot be dropped from the conditions of  \Cref{thm:circle}}
    \label{fig:drop-p1-example}
\end{figure}

To show property~\ref{P3} cannot be dropped, we give multiple examples. For the first example, let $V^{(3)}=\{v_1,v_2,v_3,v_4,v_5\}$ around a circle $C$ where $v_1,v_5$ have negative sign and others have a positive sign. Let $F^{(3)}=\{U_1,U_2\}$ where $U_1=\{v_1,v_2\}, U_2=\{v_4,v_5\}$, as shown in figure~\ref{fig:drop-p3-example}. It can be readily checked that all properties except \ref{P3} hold for this example. Observe that the only set of arcs from $V^{(3)-}$ to $V^{(3)+}$ that both $U_1,U_2$ have an incoming arc, is the set $\{(v_1,v_4), (v_5,v_2)\}$, which is not outer-planar. 
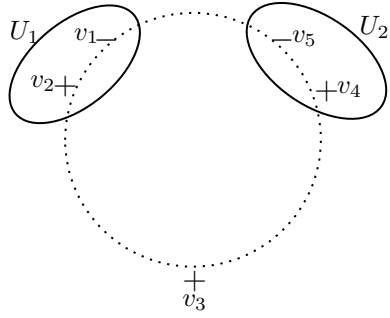
\begin{figure}
\centering
\tikzset{every picture/.style={line width=0.75pt}} 

\begin{tikzpicture}[x=0.75pt,y=0.75pt,yscale=-0.75,xscale=0.75]

\draw  [dash pattern={on 0.84pt off 2.51pt}] (100,139.91) .. controls (100,92.92) and (138.1,54.82) .. (185.09,54.82) .. controls (232.08,54.82) and (270.18,92.92) .. (270.18,139.91) .. controls (270.18,186.9) and (232.08,225) .. (185.09,225) .. controls (138.1,225) and (100,186.9) .. (100,139.91) -- cycle ;
\draw  [color={rgb, 255:red, 0; green, 0; blue, 0 }  ,draw opacity=1 ] (224.06,56.49) .. controls (233.44,43.2) and (260.39,46.09) .. (284.25,62.94) .. controls (308.11,79.79) and (319.85,104.22) .. (310.46,117.51) .. controls (301.08,130.8) and (274.13,127.91) .. (250.27,111.06) .. controls (226.41,94.21) and (214.67,69.78) .. (224.06,56.49) -- cycle ;
\draw  [color={rgb, 255:red, 0; green, 0; blue, 0 }  ,draw opacity=1 ] (145.79,56.18) .. controls (155.77,68.08) and (146.21,92.53) .. (124.44,110.81) .. controls (102.67,129.08) and (76.93,134.25) .. (66.95,122.36) .. controls (56.97,110.47) and (66.52,86.01) .. (88.29,67.74) .. controls (110.06,49.46) and (135.8,44.29) .. (145.79,56.18) -- cycle ;

\draw (264,98) node [anchor=north west][inner sep=0.75pt]   [align=left] {$\displaystyle +$};
\draw (234,64) node [anchor=north west][inner sep=0.75pt]   [align=left] {$\displaystyle -$};
\draw (117,64) node [anchor=north west][inner sep=0.75pt]   [align=left] {$\displaystyle -$};
\draw (90,95) node [anchor=north west][inner sep=0.75pt]   [align=left] {$\displaystyle +$};
\draw (104,63) node [anchor=north west][inner sep=0.75pt]  [font=\footnotesize] [align=left] {$\displaystyle v_{1}$};
\draw (75,95) node [anchor=north west][inner sep=0.75pt]  [font=\footnotesize] [align=left] {$\displaystyle v_{2}$};
\draw (176,242) node [anchor=north west][inner sep=0.75pt]  [font=\footnotesize] [align=left] {$\displaystyle v_{3}$};
\draw (279,99) node [anchor=north west][inner sep=0.75pt]  [font=\footnotesize] [align=left] {$\displaystyle v_{4}$};
\draw (249,65) node [anchor=north west][inner sep=0.75pt]  [font=\footnotesize] [align=left] {$\displaystyle v_{5}$};
\draw (61,58) node [anchor=north west][inner sep=0.75pt]  [font=\footnotesize] [align=left] {$\displaystyle U_{1}$};
\draw (294,55) node [anchor=north west][inner sep=0.75pt]  [font=\footnotesize] [align=left] {$\displaystyle U_{2}$};
\draw (176,226) node [anchor=north west][inner sep=0.75pt]   [align=left] {$\displaystyle +$};

\end{tikzpicture}
    \caption[example-drop-P3]{An example to show that \ref{P3} cannot be dropped from the conditions of \Cref{thm:circle}}
    \label{fig:drop-p3-example}
\end{figure}

The reader might think that the problem with this example is that the vertices can be decomposed into two intervals around the circle, one with all the negative signed vertices  $[v_5,v_1]$ and one with all the positive signed vertices $[v_2,v_3,v_4]$, and think that is the reason preventing us to make the desired outer-planar arc-set. Let us give another example, and in fact an infinite series of examples, that the vertices cannot be decomposed into such two intervals while all properties except \ref{P3} hold and still we don't have the desired arc-set.

To do so, this time let $V^{(3)}$ have exactly $k+1$ negative signed vertices $m_0,m_1,...,m_k$ for any $k\ge{2}$. Suppose the order of these vertices around the circle is $m_0,m_1,...,m_k$ counter clockwise. Going around the circle counter clockwise, suppose $p_0$ is a positive signed vertex immediately before $m_0$ and $p_k\neq{p_0}$ is a positive signed vertex immediately after $m_k$. Suppose in the interval from $p_k$ to $m_1$, there is at least one positive signed vertex other than $p_0,p_k$. For any $1\le{i}\le{k-1}$, put an arbitrary number, possibly $0$, of positive signed vertices in the interval going from $m_i$ to $m_{i+1}$ counter clockwise. Let $U_0=\{m_0,p_0\}$, and for $i\in\{1,2,...,k\}$ let $U_i$ contain the vertices from $m_i$ to $p_k$, both inclusive and going counter clockwise. Finally $F^{(3)}=\{U_0, U_1, ..., U_k\}$. Figure~\ref{fig:drop-p3-example-2} shows the general example on the left, as well as the ones with minimum number of vertices next to it.  \\
Observe that since there is at least one positive signed vertex other than $p_k,p_0$ going counter clockwise from $p_k$ to $m_1$, then $U_0\cup{U_1}\neq{V^{(3)}}$, and in fact there are no two sets whose their union is $V^{(3)}$, and hence \ref{P4} holds. The reader can check that the example has all the other properties except \ref{P3}. We try to construct the desired arc-set. Observe that for the set $U_0$, the arc must come from vertex $m_j$ for some $j\in\{1, 2, ..., k\}$ to $p_0$. However, then the set $U_j$ cannot receive any arc without crossing the arc $(m_j,p_0)$. 
Therefore, there is no outer-planar arc-set from $V^{(3)-}$ to $V^{(3)+}$ for this example such that every set in $F^{(3)}$ receives an arc.

\begin{figure}
\centering
\input{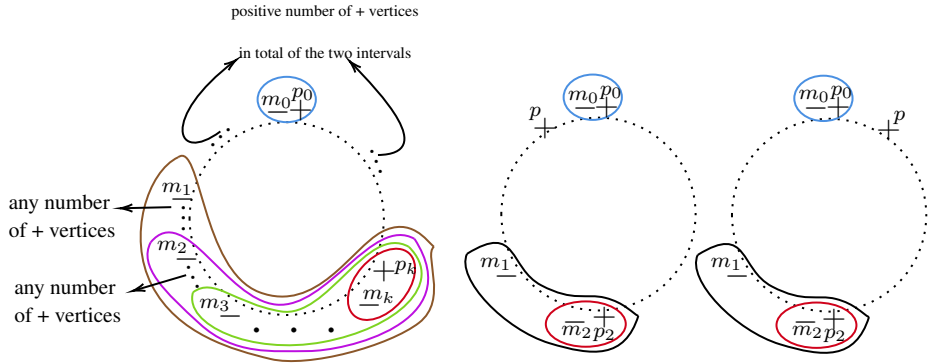}
    \caption[example-drop-P3-2]{On the left, An infinite series of examples to show that \ref{P3} cannot be dropped from conditions of \Cref{thm:circle}. Middle and right pictures show the example for $k=2$ with minimum number of vertices (6). Their only difference is the position of $p$. }
    \label{fig:drop-p3-example-2}
\end{figure}

Finally, to show that the property~\ref{P4} cannot be dropped, take any example $(V^{(3)},\sigma^{(3)},F^{(3)})$ we mentioned for dropping~\ref{P3} and do the following: Keep the same vertex set, that is $V^{(4)}=V^{(3)}$, but switch the sign of every vertex, that is  $\sigma^{(4)}=\bar{\sigma}^{(3)}$. Furthermore, let $F^{(4)}=\{V^{(4)}\backslash{U}|U\in{F^{(3)}}\}$, that is the sets in $F^{(4)}$ come from complementing the sets in $F^{(3)}$. We leave this to the reader to show that the example $(V^{(4)},\sigma^{(4)}, F^{(4)})$ has all the properties except~\ref{P4} and there is no outer-planar arc-set from $V^{(4)-}$ to $V^{(4)+}$ that each set in $F^{(4)}$ gets an incoming arc. 

\subsection{Algorithm}\label{sec:circle-algorithm}
In the following, we provide a full overview of our algorithm to solve the circle problem. We provide the lemmas (and proofs) that justify the different reduction steps our algorithm performs in~\Cref{sec:circle-lemmas}. \Cref{thm:circle} is proved in~\Cref{proof-of-circle-thm}. This algorithm can be executed in $O(n^2m^2)$ where $n=|V|$ and $m=|F|$. \\
\textbf{Algorithm}\\
\textbf{Input:} an instance of the circle problem $(V,\sigma,F)$ around a circle $C$ \\
\textbf{Output:} an outer-planar arc subset $E$ from $V^-$ to $V^+$ that covers $F$
\begin{enumerate}[label=\arabic*.]
        \item\label{Stepadjacent}
        Let $E$ be the set of all the arcs between adjacent vertices with opposite signs, oriented from the $-$ signed vertex to the $+$ signed one. Moreover, 
        let $F_1\subseteq{F}$  be the sets not covered by $E$ and let $V_1=V$ (\Cref{lem:adjacent-arcs}). 
        \item\label{Stepminusminus} While there exists $W^*\in{F_1}$ with a $-$ signed end-vertex whose adjacent out-neighbor has a $+$ sign: \\
        \begin{enumerate}
            \item If $W^*$ is not a co-$2$-block, let $m$ be its $-$ signed end-vertex with
            $+$ signed  adjacent out-neighbor $p_1$. Let $p_1, p_2, ..., p_t$ for some $t\ge{1}$ be the vertices of the sign-block of $p_1$ in this order and let $m'$ be the $-$ signed vertex adjacent to $p_t$  (\Cref{lem:ending-minus-out-neighbor} and figure~\ref{fig:alg-steps-combined}.a). Then update: $E \leftarrow E\cup\{(m, p_i)\mid i\in[t]\}\cup\{(m',p_t)\}$, $F_1\leftarrow\{U\in{F_1}\mid U\text{ not covered by }E\}$, $V_1\leftarrow{V_1\backslash\{p_1,p_2,...p_t\}}$ and $F_1\leftarrow F_1[V_1]$. 
            \item If $W^*$ is a co-$2$-block, let $m_1$ be its $-$ signed end-vertex with $+$ signed adjacent out-neighbor $p$. Let $m_1, m_2, ...,m_\ell$ for some $\ell\ge{1}$ be the vertices of the sign-block of $m_1$ in this order and let $p'$ be the $+$ signed vertex adjacent to $m_\ell$ (\Cref{lem:co-2-block-ending-minus} and figure~\ref{fig:alg-steps-combined}.b). Then update: $E\leftarrow E\cup\{(m_i, p)|i\in[\ell]\}\cup\{(m_\ell,p')\}$, $F_1\leftarrow\{U\in{F_1}|U\text{ not covered by }E\}$, $V_1\leftarrow V_1 \backslash \{m_1, m_2,...m_\ell\}$ and $F_1\leftarrow F_1[V_1]$.  
        \end{enumerate}
        \item\label{step-mmpp-all} If consecutive vertices $m_2,m_1,p_1,p_2\in{V_1}$ with this order on $C$ satisfy $\sigma(m_1)=\sigma(m_2)=-$ and $\sigma(p_1)=\sigma(p_2)=+$, there are three cases: 
        \begin{enumerate}
            \item\label{removablecase} One of $m_1$ or $p_1$ is \textit{removable}, i.e., either there are no sets $U,W\in{F_1}$, possibly equal, such that $m_1$ is the only negative vertex of $V_1\backslash{(U\cup{W})}$, or there are no sets $U,W\in{F_1}$, possibly equal, such that $p_1$ is the only positive vertex of $U\cap{W}$. Let $u\in\{m_1,p_1\}$ be the removable vertex. Update $V_1\leftarrow{V_1\backslash\{u\}}$ and $F_1\leftarrow{F_1[V_1]}$. Add a new arc to $E$ between the new adjacent opposite signed vertices ($m_2, p_1$ if $u=m_1$ or $m_1,p_2$ if $u=p_1$) and remove the covered sets from $F_1$ (Lemmas~\ref{lem:m-m-p-p-set-set}, \ref{lem:m-m-p-p-pair-pair} and \ref{lem:m-m-p-p-simple}). 
            \item\label{twoblockcase} $\{m_1,p_1\}\in{F_1}$, and there exist disjoint sets $U,W\in{F_1}$ where $m_1$ is the only $-$ signed vertex in $V_1\backslash(U\cup{W})$, let $T$ be the set of all such pairs $(U,W)$.  
         For each pair exactly one set, say $W$, excludes $m_2$ (\Cref{lem:m-m-p-p-2-block}).  
        Let $p_1, p_2, ..., p_t$ for some $t\ge{2}$ be the vertices of the sign-block of $p_1$ in this order. Then some $i\in\{2,...,t\}$ satisfies $p_i\in{W}$ and $p_{i-1}\notin{W}$ (\Cref{lem:m-m-p-p-2-block}).  Denote that as $f(U,W)=i$. Choose the pair maximizing $f(U,W)$ with value $j$. See  figure~\ref{fig:alg-steps-combined}.c. Then update: $E\leftarrow E\cup\{(m_2,p_k)|k\in[j]\}\cup\{(m_1,p_1)\}$, $F_1\leftarrow\{U\in{F_1}|U\text{ not covered by }E\}$, $V_1\leftarrow V_1 \backslash \{m_1, p_1, p_2, ..., p_{j-1}\}$, and $F_1\leftarrow F_1[V_1]$.
        \item\label{cotwoblockcase} $V_1\backslash\{m_1,p_1\}\in{F}$, and there exists sets $U,W\in{F_1}$ such that $U\cup{W}=V_1$ and $p_1$ is the only $+$ signed vertex in $U\cap{W}$, let $T$ be the set of all such pairs $(U,W)$. For each pair exactly one set, say $W$ contains $p_2$ (\Cref{lem:m-m-p-p-co-2-block}). Let $m_\ell, m_{\ell-1}, ...,m_2, m_1$ for some $\ell\ge{2}$ be the vertices of the sign-block of $m_1$ in this order. Then some $i\in\{2, ..., \ell\}$ satisfies $m_{i-1}\in{W}$ and $m_{i}\notin{W}$ (\Cref{lem:m-m-p-p-co-2-block}). Denote that as $f(U,W)=i$. Choose the pair minimizing  $f(U,W)$ with value $j$. See figure~\ref{fig:alg-steps-combined}.d. Then update: $E\leftarrow E\cup\{(m_k,p_2)|k\in[j]\}\cup\{(m_1,p_1)\}$, $F_1\leftarrow\{U\in{F_1}|U\text{ not covered by }E\}$, $V_1\leftarrow V_1 \backslash \{p_1, m_1, m_2, ..., m_{j-1}\}$ and $F_1\leftarrow F_1[V_1]$. 
        \end{enumerate}
        Go to Step~\ref{Stepminusminus}.
        \item\label{step-mmpm} If consecutive vertices $m_2, m_1, p,m_3\in{V_1}$ with this order on $C$ satisfy $\sigma(m_1)=\sigma(m_2)=\sigma(m_3)=-$ and $\sigma(p)=+$, then $m_1$ is removable (\Cref{lem:m-m-p-m}). Update $V_1\leftarrow{V_1\backslash\{m_1\}}$, $E\leftarrow{E\cup\{(m_2, p)\}}$, $F_1\leftarrow{F_1[V_1]}$ and remove the covered sets from $F_1$. Go to Step~\ref{Stepminusminus}.
        \item\label{step-ppmp} If consecutive vertices $p_2, p_1, m, p_3\in{V_1}$ with this order on $C$ satisfy $\sigma(p_1)=\sigma(p_2)=\sigma(p_3)=+$ and $\sigma(m)=-$, then $p_1$ is removable (\Cref{lem:p-p-m-p}). Update $V_1\leftarrow{V_1\backslash\{p_1\}}$, $E\leftarrow{E\cup\{(m, p_2)\}}$, $F_1\leftarrow{F_1[V_1]}$ and remove the covered sets from $F_1$. Go to Step~\ref{Stepminusminus}. 
        \item\label{Stepfinal} Output $E$. 
    \end{enumerate} 
\begin{figure}[!hbt]
    \centering
    \input{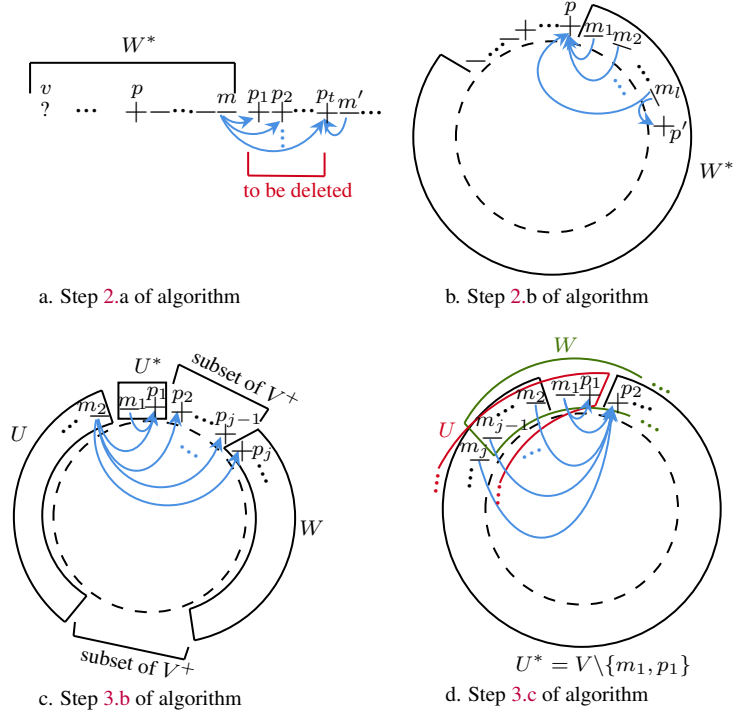}
    \caption[]{Steps of the Algorithm for \Cref{thm:circle}. 
    }
    \label{fig:alg-steps-combined}
\end{figure}
\paragraph{Running Time Analysis.} We analyze the running time based on $n=|V|$ and $m=|F|$, assuming $n<m$. Steps~\ref{Stepadjacent}, \ref{step-mmpm}, \ref{step-ppmp} takes $O(n+m)$ to find the desired consecutive vertices and to update the family $F$ or $F_1$. Step~\ref{Stepadjacent} preserves $V$, but Steps~\ref{step-mmpm} and \ref{step-ppmp} each remove one vertex, so they execute at most $n$ times. Therefore, the total time required for these three steps is $O(n(n+m))=O(nm)$. Step~\ref{Stepminusminus} takes $O(m+n)$ to find $W^*$, check if it is a co-$2$-block, and $O(m+n)$ to update $F$ and $V$. Since it removes at least one vertex, it executes at most $n$ times for the total time $O(n(n+m))=O(nm)$. For Step~\ref{step-mmpp-all}, it takes $O(n)$ to find the desired consecutive vertices. Then $O(m^2)$ to enumerate pairs of sets in the family, which finishes \ref{removablecase} or constructs $T$ for \ref{twoblockcase} and  \ref{cotwoblockcase}. Then it takes $|T|O(n)$ to find the pair with maximum/minimum $f$ where $|T|\le{m \choose 2}$; Hence it takes $O(nm^2)$ to finish any case of Step~\ref{step-mmpp-all}. Step~\ref{removablecase} removes one vertex, and Steps~\ref{twoblockcase},\ref{cotwoblockcase} remove at least two vertices $m_1,p_1$, hence Step~\ref{step-mmpp-all} can be executed at most $n$ times. Therefore, the total time required for this step is $O(n^2m^2)$.  Summing up, the algorithm can be executed in time $O(n^2m^2)$.

\subsection{Auxiliary Statements}\label{sec:circle-lemmas}
In this section, we state and prove the lemmas needed to prove the correctness of our algorithm to solve the circle problem and to derive \Cref{thm:circle}.

\begin{proposition}\label{rem:plus-in-intersection}
  Let $(V,\sigma,F)$ be an instance of the circle problem. Let $U,W\in{F}$ such that $U\cap{W}\neq{\emptyset}$. Then $(U\cap{W})^+\neq{\emptyset}$.   
\end{proposition}
\begin{proof}
    If $U\cup{W}\neq{V}$, then by \ref{P1}, $U\cap{W}\in{F}$ and by \ref{P2}, $(U\cap{W})^+\neq{\emptyset}$. Otherwise, $U\cup{W}=V$ and by \ref{P4}, we have $(U\cap{W})^+\neq{\emptyset}$. 
\end{proof}

\begin{proposition}\label{rem:minus-out-union}
Let $(V,\sigma,F)$ be an instance of the circle problem. Let $U,W\in{F}$ such that $U\cup{W}\neq{V}$. Then $(U\cup{W})^{c-}\neq{\emptyset}$.
\end{proposition}
\begin{proof}
    If $U\cap{W}\neq{\emptyset}$, then by \ref{P1}, $U\cup{W}\in{F}$ and by \ref{P2}, $(U\cup{W})^{c-}\neq{\emptyset}$. Otherwise, $U\cap{W}=\emptyset$ and by \ref{P3}, we have $(U\cup{W})^{c-}\neq{\emptyset}$. 
\end{proof}
\LemmaDeleteCoveredSets*
\begin{proof}
    Clearly $F'$ has properties \ref{P0}, \ref{P2}, \ref{P3} and \ref{P4}. So it remains to show that $F'$ is crossing. Let $U,W\in{F'}$ such that $U\cap{W}\neq\emptyset$ and $U\cup{W}\neq{V}$. Since $F$ is crossing, $U\cap{W}, U\cup{W}\in{F}$. Suppose $U\cap{W}\notin{F'}$. That means $U\cap{W}$ was covered by some arc $(m,p)\in{E}$ where $m\notin{(U\cap{W})}$ has a $-$ sign and $p\in{(U\cap{W})}$ has a $+$ sign. Then w.l.o.g. we may assume $m\notin{U}$. Since $p\in{U\cap{W}}\subseteq{U}$, then $U$ is also covered by $(m,p)\in{E}$ and hence $U\notin{F'}$, a contradiction. This proves $U\cap{W}\in{F'}$. \\
     Now, suppose $U\cup{W}\notin{F'}$. That means $U\cup{W}$ was covered by some arc $(m',p')\in{E}$ where $m'\notin{(U\cup{W})}$ has a $-$ sign and $p'\in{U\cup{W}}$ has a $+$ sign. Then w.l.o.g. we may assume $p'\in{W}$. Since $m'\notin{(U\cup{W})}$, we have $m'\notin{W}$. Hence $W$ is also covered by $(m',p')\in{E}$ and so $W\notin{F'}$, a contradiction. This shows $U\cup{W}\in{F'}$. \\
    As a consequence of the two arguments above, $F'$ is crossing.
\end{proof}

\begin{lemma}\label{lem:adjacent-arcs}
Let $(V,\sigma,F)$ be an instance of the circle problem. Let $A$ be the set of all directed arcs between adjacent opposite-signed vertices from the $-$ signed vertex to the $+$ signed one. Let $F'=F\backslash{F_A}$ consists of all the sets in $F$ that are not covered by $A$. Then $(V,\sigma,F')$ is an instance of the circle problem and if a set of outer-planar arcs $E'$ from $V^-$ to $V^+$ is a solution for $(V,\sigma,F')$, then $E'\cup{A}$ is a solution for $(V,\sigma,F)$. 
\end{lemma}
\begin{proof}
$(V,\sigma,F')$ is an instance of the circle problem as a direct consequence of \Cref{lem:covering-arcs}. Arcs in $A$ do not separate $V$ around the circle, therefore $E'\cup{A}$ is also outer-planar. Since $A$ covers $F_A$ and $E'$ covers $F'=F\backslash{F_A}$, then $E'\cup{A}$ covers $F$. 
\end{proof}

\begin{lemma}\label{lem:plus-end-plus-out}
    Let $(V,\sigma,F)$ be an instance of the circle problem. Then we may assume for any set $W\in{F}$ that ends with a $+$ signed vertex $p$, $p$'s adjacent out-neighbor(s) with respect to $W$ is also $+$ signed.
\end{lemma}
\begin{proof}
    This is a direct use of \Cref{lem:adjacent-arcs}. 
\end{proof}

\begin{definition}
We define an additional property for an instance $(V,\sigma,F)$ of the circle problem: 
\begin{enumerate}[leftmargin=3em, label={\bfseries AP\arabic*}, start=1]
    \item\label{AP1} For every $U\in{F}$ with a $+$ signed end vertex $v$, $v'$s adjacent out-neighbor(s) with respect to $U$ is $+$ signed.
\end{enumerate}
\end{definition}

\begin{lemma}\label{lem:P0-remains}
Let $(V,\sigma,F)$ be an instance of the circle problem. Let $V'$ be derived from $V$ by deleting a set of consecutive vertices and let $F'=F[V']$. Let $\sigma'=\sigma[V']$.Then \ref{P0} holds for $(V',\sigma',F')$. 
\end{lemma}
\begin{proof}
    Since we are deleting a consecutive interval of vertices, every set in $F'$ remains an interval on the circle, therefore \ref{P0} holds for $(V',\sigma',F')$.
\end{proof}

\begin{lemma}\label{lem:crossing-family-subset}
Suppose $F$ is a crossing family on the ground set $V$. Then for any subset $V'\subseteq{V}$, $F[V']$ is also a crossing family. 
\end{lemma}
\begin{proof}
    This is clear by definition of crossing families. 
\end{proof}

\begin{lemma}\label{lem:P0-P1-remains}
Let $(V,\sigma,F)$ be an instance of the circle problem. Let $V^d$ be a set of consecutive vertices and let $A$ be any set of arcs. Let $F'=(F\backslash{F_A})[V\backslash{V^d}]$. Then \ref{P0} and \ref{P1} hold for $F'$.   
\end{lemma}
\begin{proof}
    By \Cref{lem:P0-remains}, \ref{P0} holds for $F'$. Using Lemmas~\ref{lem:covering-arcs} and \ref{lem:crossing-family-subset}, $F'$ is a crossing family over the ground set $V\backslash{V^d}$. 
\end{proof}

\begin{lemma}\label{lem:union-covering}
  Let $(V,\sigma,F)$ be an instance of the circle problem. Let $V^d$ be a set of consecutive vertices and let $A$ be a set of arcs covering any set $W\in{F}$ with $W\subseteq{V^d}$. If $E$ covers $F'=(F\backslash{F_A})[V\backslash{V^d}]$, then $E\cup{A}$ covers $F$. 
\end{lemma}
\begin{proof}
    We need to show that $E\cup{A}$ covers any set $U\in{F}$. If $U$ is covered by some arc in $A$, then we are done. Otherwise, $U\in{F\backslash{F_A}}$. Since $U$ is not covered by $A$, then $U\backslash{V^d}\neq{\emptyset}$. Hence $U[V\backslash{V^d}]\in{F'}$ and is covered by an arc $e\in{E}$. Thus, $U$ is covered by $e$. 
    This shows that $E\cup{A}$ covers $F$. 
\end{proof}

\begin{lemma}\label{lem:ending-minus-out-neighbor}
Let $(V,\sigma,F)$ be an instance of the circle problem. Let $W^*\in{F}$ be a set with a $-$ signed end vertex $m$ whose adjacent out-neighbor $p_1$ with respect to $W^*$, is $+$ signed. Suppose $W^*$ is not a co-$2$-block. Traversing the vertices on the circle from $m$ to the next negative vertex $m'$ via $p_1$, let $p_2, ..., p_t$ be the sequence of positive vertices visited for some $t\ge 1$. Let $A=\{(m,p_i)|i\in[t]\}\cup\{(m',p_t)\}$. Let $V'=V\backslash\{p_1,p_2,..., p_t\}$ and $\sigma'=\sigma[V']$. Finally let $F'=(F\backslash{F_A})[V']$. Then $(V',\sigma',F')$ is an instance of the circle problem and if $E'$ is a solution for $(V',\sigma',F')$, then $E'\cup{A}$ is a solution for $(V,\sigma,F)$. 
\end{lemma}

\begin{proof}
Figure~\ref{fig:alg-steps-combined}.a  shows $W^*$ and $A$. W.l.o.g we may assume $W^*$ is the smallest set in $F$ with the end vertex $m$ such that it does not include $p_1$.  
Using \Cref{lem:P0-P1-remains}, properties \ref{P0} and \ref{P1} hold for $(V',\sigma',F')$. 

To show \ref{P2}, let $W\in{F}$ that is not covered by $A$ such that $W'=W[V']\in{F'}$. Since we are only deleting $+$ signed vertices from $V$, then $W'^{c-}=W^{c-}\neq{\emptyset}$. Now, for the sake of contradiction, assume $W'^+=\emptyset$. This means $W^+\subseteq\{p_1, p_2, ..., p_{t}\}$. Since $W$ is not covered by $A$, we have $m\in{W}$. Then $W\cap{W^*}\neq{\emptyset}$ and by \Cref{rem:plus-in-intersection}, there exists a $+$ signed vertex in $W\cap{W^*}$. This means that $W^*\cap\{p_1, p_2, ..., p_t\}\neq\emptyset$, a contradiction. This concludes \ref{P2} for $(V',\sigma',F')$. 

To show \ref{P3} for $F'$, let $C,D\in{F}$ that are not covered by $A$ such that $C'=C[V']\in{F'}, D'=D[V']\in{F'}$, and $C'\cap{D'}=\emptyset$. First suppose $C\cap{D}\neq{\emptyset}$. Since $C'\cap{D'}=\emptyset$, then $C\cap{D}\subseteq\{p_1,p_2, ..., p_{t}\}$. Thus $m\notin{C\cap{D}}$ and w.l.o.g we may assume $m\notin{C}$. Then $C$ would be covered by arc $(m, p_k)\in{A}$ for some $k\in[t]$ such that $p_k\in{C}$ which is not the case. Hence $C\cap{D}=\emptyset$. Then since we only deleted $+$ signed vertices from $V$, we have $(C'\cup{D'})^{c-}= (C\cup{D})^{c-}$ which is non-empty by \ref{P3} for $F$. This concludes \ref{P3} for $F'$.

To prove \ref{P4} for $F'$, let $U,W\in{F}$ that are not covered by $A$ such that $U'=U[V']\in{F'}, W'=W[V']\in{F'}$ and $U'\cup{W'}=V'$. First suppose $U\cup{W}\neq{V}$. Then $\emptyset\neq(U\cup{W})^c\subseteq\{p_1, p_2, ..., p_t\}$, a contradiction to \Cref{rem:minus-out-union}. Hence $U\cup{W}=V$. Now for the sake of contradiction, assume $(U'\cap{W'})^+=\emptyset$. Thus $(U\cap{W})^+\subseteq\{p_1, p_2, ..., p_t\}$ and is non-empty by \ref{P4} for $F$. Note that since $U,W$ are not covered by $A$, $m\in{U,W}$. 
Let $p$ be the first $+$ signed vertex inside $W^*$ after the sign-block of $m$. Note that $p$ exists by \ref{P2} for $F$. Since $(U\cap{W})^+\subseteq\{p_1, p_2, ..., p_t\}$, we may assume $p\notin{W}$. Note that $W\cap{W^*}$ is non-empty as $W,W^*$ include $m$. If $W\cup{W^*}\neq{V}$, since $F$ is a crossing family, $W\cap{W^*}\in{F}$ and forms an interval on the circle. Since $m\in{W,W^*}$ and $p\notin{W}$ and $p_1\notin{W^*}$, then $W\cap{W^*}\subseteq{\text{sign-block}[m]}$, a contradiction to \Cref{rem:plus-in-intersection}. Hence $W\cup{W^*}=V$. 

Since $W\cap{W^*}\neq{\emptyset}$ and by \Cref{rem:plus-in-intersection}, there exist a $+$ signed vertex $p'\notin\{p_1, p_2, ..., p_t\}$ such that $p'\in{W\cap{W^*}}$. Note that $p'\notin{U}$ because $(U\cap{W})^+\subseteq\{p_1, p_2, ..., p_t\}$. Since $p'\notin{U}$, $W^*$ is not a subset of $U$. Recall that $W^*\cap{U}\neq{\emptyset}$ as $U,W^*$ include $m$. Suppose $W^*\cup{U}\neq{V}$. Since $F$ is crossing, $W^*\cap{U}\in{F}$. Then $W^*\cap{U}$ is a smaller set compared to $W^*$ that contains $m$ and not $p_1$, a contradiction to the assumption we had for $W^*$ at the beginning of the proof. Therefore $W^*\cup{U}=V$. 

Let $v\neq{m}$ be the other end vertex of $W^*$. Let $I=W^{*c}$. Since $W^*\cup{U}=V$ and $W\cup{W^*}=V$, then $I\subset{W,U}$. Since $(W\cap{U})^+\subseteq\{p_1, p_2, ..., p_t\}$, we have $I^+=\{p_1, p_2, ..., p_t\}$. Hence $I$ consists only of two blocks: the plus signed block $\{p_1, p_2, ..., p_t\}$ and the sign-block of $m'$. This means $W^*$ is a co-$2$-block, a contradiction to the lemma's assumption. This concludes the proof of \ref{P4} for $F'$.  

We showed that $(V',\sigma',F')$ is also an instance of the circle problem. Then if $E'$ is a solution to $(V',\sigma',F')$, by construction $E'\cup{A}$ is outer-planar. 
For any set $W\in{F}$ such that $W\subseteq{V\backslash{V'}}$, $W$ is covered by an arc $(m, p_k)\in{A}$ for some $k\in[j]$ such that $p_k\in{W}$. Hence by \Cref{lem:union-covering}, $E'\cup{A}$ covers $F$. Therefore, $E'\cup{A}$ is a solution to $(V,\sigma, F)$. 
\end{proof}

\begin{lemma}\label{lem:co-2-block-ending-minus}
Let $(V,\sigma,F)$ be an instance of the circle problem with additional properties \ref{AP1}. Let $W^*\in{F}$ be a co-$2$-block ending with a $-$ signed vertex $m_1$. Suppose $m_1$'s adjacent out-neighbor with respect to $W^*$ is a $+$ signed vertex $p$. Let $m_1, m_2, ..., m_\ell\in{W^*}$ be the sign-block of $m_1$ and let $p'$ be the $+$ signed vertex next to $m_\ell$. Let $A=\{(m_i, p)|i\in[\ell]\}\cup{(m_\ell, p')}$. Let $V'=V\backslash\{m_1, m_2, ..., m_\ell\}$ and $\sigma'=\sigma[V']$. Finally let $F'=(F\backslash{F_A})[V']$. Then $(V',\sigma',F')$ is an instance of the circle problem and if $E'$ is a solution for $(V',\sigma',F')$, then $E'\cup{A}$ is a solution for $(V,\sigma,F)$. 
\end{lemma}

\begin{proof}
Figure~\ref{fig:alg-steps-combined}.b shows $W^*$ and $A$. We first show that $W^*$ is not a $2$-block. Suppose otherwise and let $v\neq{m_1}$ be the other end vertex of $W^*$. Then $v$ has a $+$ sign. Since $W^*$ is a co-$2$-block and $p$, the adjacent out-neighbor of $m_1$ with respect to $W^*$, has a $+$ sign, then the adjacent out-neighbor of $v$ with respect to $W^*$ has a $-$ sign, a contradiction to property \ref{AP1}. Hence $W^*$ is not a $2$-block and so $W^{*c}$ is not a co-$2$-block.  Then the lemma is a consequence of applying \Cref{lem:ending-minus-out-neighbor} for the dual instance $(V, \bar{\sigma}, \bar{F})$ and set $W^{*c}\in{\bar{F}}$. 
\end{proof}

\begin{lemma}\label{lem:ends-minus-minus-out}
Let $(V,\sigma,F)$ be an instance of the circle problem with additional property \ref{AP1}. Then we may assume for any set $W\in{F}$ that ends with a $-$ signed vertex $m$, $m$'s adjacent out-neighbor with respect to $W$ is also $-$ signed.
\end{lemma}

\begin{proof}
This is a consequence of applying \Cref{lem:ending-minus-out-neighbor} if $W$ is not a co-$2$-block and \Cref{lem:co-2-block-ending-minus} otherwise. 
\end{proof}

\begin{definition} We define the following additional property for an instance $(V,\sigma,F)$ of the circle problem:
\begin{enumerate}[leftmargin=3em, label={\bfseries AP\arabic*}, start=2]
    \item\label{AP2} For every $U\in{F}$ with a $-$ signed end vertex $v$, $v'$s adjacent out-neighbor(s) with respect to $U$ is $-$ signed.
\end{enumerate} 
\end{definition}

\begin{definition}\label{dangerous-sets}
Let $(V,\sigma,F)$ be an instance of the circle problem.
For a $+$ signed vertex $v\in{V}$, we say $U\in{F}$ is a \normalfont{$v$-dangerous-set} if $U^+=\{v\}$. Analogously, for a $-$ signed vertex $u\in{V}$, $W\in{F}$ is a $u$-dangerous-set if $W^{c-}=\{u\}$. 
\end{definition}

\begin{definition}\label{dangerous-pairs}
Let $(V,\sigma,F)$ be an instance of the circle problem. For a $+$ signed vertex $v\in{V}$, we say $U,W\in{F}$ is a \normalfont{$v$-dangerous-pair} if $U\cup{W}=V$ and $(U\cap{W})^+=\{v\}$. \\
Moreover, for a $-$ signed vertex $u\in{V}$. $A,B\in{F}$ is a $u$-dangerous-pair if $U\cap{W}=\emptyset$ and $(A\cup{B})^{c-}=\{u\}$. 
\end{definition}

Observe that dangerous sets are related to property \ref{P2} and dangerous pairs are corresponding to properties \ref{P3} and \ref{P4}. 

\begin{remark}\label{rem:deletable-vertex}
Let $(V,\sigma,F)$ be an instance of the circle problem. Let $v\in{V}$ such that there is no $v$-dangerous set or pair in $F$. Then $(V\backslash\{v\},\sigma[V\backslash\{v\}], F[V\backslash\{v\}])$ is an instance of the circle problem and if $E$ covers $F[V\backslash\{v\}]$, $E$ covers $F$.     
\end{remark}

\begin{definition}\label{deletable-vertex}
For an instance of the circle problem $(V,\sigma,F)$, we call $v\in{V}$ \normalfont{removable} if there is no $v$-dangerous set or pair in $F$. 
\end{definition}

\begin{lemma}\label{lem:m-m-p-p-set-set}
    Let $(V,\sigma,F)$ be an instance of the circle problem with additional property \ref{AP1}. Suppose four consecutive vertices $m_2,m_1,p_1,p_2$ have signs $-, -,$\\$ +, +$ respectively. Then there do not exist sets $U_1, U_2\in{F}$ such that $U_1$ is a $m_1$-dangerous set and $U_2$ is a $p_1$-dangerous set.
\end{lemma}

\begin{proof}
    Suppose otherwise. Then $U_1^{c-}=\{m_1\}$ and by \ref{AP1}, $p_1\notin{U_1}$. Moreover, $U_2^+=\{p_1\}$, and so $p_2\notin{U_2}$ and by \ref{AP1}, $m_1\in{U_2}$. If $m_2\notin{U_2}$, then $U_1\cap{U_2}=\emptyset$, and since $U_1\cup{U_2}$ covers all $-$ signed vertices, hence $(U_1\cup{U_2})^c\subset{V^+}$, a contradiction to property \ref{P3} of $F$. Therefore, we have $m_2\in{U_2}$. Then $U_1\cap{U_2}\neq{\emptyset}$. Hence by \Cref{rem:plus-in-intersection}, there exist a $+$ signed vertex $v\in{U_1\cap{U_2}}$. However, since $p_1\notin{U_1}$, we have a plus signed $v\neq{p_1}\in{U_2}$, a contradiction.
\end{proof}

\begin{lemma}\label{lem:m-m-p-p-pair-pair}
    Let $(V,\sigma,F)$ be an instance of the circle problem with additional property \ref{AP1}. Suppose four consecutive vertices $m_2, m_1, p_1, p_2$ have signs $-,-,$\\$+,+$ respectively. Then there do not exist sets $W_1, W_2, U_1, U_2\in{F}$ such that $(W_1, W_2)$ is a $m_1$-dangerous pair and $(U_1,U_2)$ is a $p_1$-dangerous pair.
\end{lemma}
\begin{proof}
Suppose otherwise. In particular, $W_1\cap{W_2}=\emptyset$ and $(W_1\cup{W_2})^{c-}=\{m_1\}$. Then we may assume $m_2\in{W_1}$ and by \ref{AP1}, we have $p_1\notin{W_1, W_2}$. If $p_2\in{W_1}$, since $W_1$ is an interval on the circle not including $\{m_1, p_1\}$, we have $W_1=V\backslash\{m_1, p_1\}$ and since $m_1, p_1\notin{W_2}$, we get $W_2\subseteq{W_1}$, a contradiction to $W_1\cap{W_2}=\emptyset$. Hence $p_2\notin{W_1}$.

We have $U_1,U_2$ as a $p_1$-dangerous pair, i.e. $U_1\cup{U_2}=V$ and $(U_1\cap{U_2})^+=\{p_1\}$. By \ref{AP1}, we have $m_1\in{U_1, U_2}$. 
Then we may assume $p_2\notin{U_1}$ and $p_2\in{U_2}$. 
Since $p_2\notin{W_1}$, we have $U_1\cup{W_1}\neq{V}$. If $m_2\notin{U_1}$, then $U_1=\{m_1, p_1\}$ and since $U_1\cup{U_2}=V$ and $m_1, p_1\in{U_2}$, we get $U_2=V$, a contradiction to \ref{P2}. Hence $m_2\in{U_1}$
and so $W_1\cap{U_1}\neq{\emptyset}$. Thus, since $F$ is crossing, $W_1\cup{U_1}\in{F}$. 

First assume $W_1$ is not a subset of $U_1$. Then $W_1\cup{U_1}=W_1\cup\{m_1, p_1\}$ as they are both intervals. Since $W_2\cap{W_1}=\emptyset$ and $m_1, p_1\notin{W_2}$, we have $W_2\cap{(W_1\cup{U_1})}=\emptyset$. However, $(W_2\cup{W_1\cup{U_1}})^c=(W_2\cup{W_1\cup\{m_1,p_1\}})^c\subset{V^+}$ because $(W_1\cup{W_2})^{c-}=\{m_1\}$, a contradiction to \ref{P3} for $W_2\in{F}$ and $W_1\cup{U_1}\in{F}$. This means $W_1\subset{U_1}$. 

Since $W_1\subset{U_1}$, we have $(W_2\cup{U_1})^c\subseteq{(W_2\cup{W_1})^c\backslash\{m_1, p_1\}}\subset{V^+}$; thus by \ref{P3}, $W_2\cap{U_1}\neq{\emptyset}$ and by \Cref{rem:minus-out-union}, $W_2\cup{U_1}=V$. Therefore, since $p_2\notin{U_1}$, we have $p_2\in{W_2}$ and since $p_2\in{U_2}$, we get $U_2\cap{W_2}\neq{\emptyset}$. 

By \ref{P2}, there exist a $+$ signed vertex $p\in{W_1}$ and since $p_1\notin{W_1}$, $p\neq{p_1}$. Since $W_1\subset{U_1}$, we have $p\in{U_1}$ and since $(U_1\cap{U_2})^+=\{p_1\}$, we have $p\notin{U_2}$. As $p\in{W_1}$ and $W_1\cap{W_2}=\emptyset$, we have $p\notin{W_2}$ and hence $U_2\cup{W_2}\neq{V}$. 

Since $F$ is crossing, $U_2\cup{W_2}\in{F}$. Now suppose $(U_2\cup{W_2})\cap{W_1}\neq{\emptyset}$. Then by \Cref{rem:plus-in-intersection}, there exist a $+$ signed vertex $p'\in{(U_2\cup{W_2})\cap{W_1}}$ and since $p_1\notin{W_1}$, we have $p'\neq{p_1}$. Since $W_1\cap{W_2}=\emptyset$, $p'\in{U_2}$. Moreover, $p'\in{W_1}\subset{U_1}$. Hence $p'\in{U_1\cap{U_2}}$, a contradiction to $(U_1\cap{U_2})^+=\{p_1\}$. Therefore we have $(U_2\cup{W_2})\cap{W_1}={\emptyset}$. However, $((U_2\cup{W_2})\cup{W_1})^c\subseteq{(W_1\cup{W_2}\cup\{m_1, p_1\})^c\subset{V^+}}$, a contradiction to \ref{P3}. As a result, an $m_1$-dangerous pair and an $p_1$-dangerous pair cannot exist simultaneously. 
\end{proof}

\begin{lemma}\label{lem:m-m-p-p-simple}
    Let $(V,\sigma,F)$ be an instance of the circle problem with additional property \ref{AP1}. Suppose four consecutive vertices $m_2,m_1, p_1,p_2$ have signs $-, -,$\\$+, +$ respectively, and $\{m_1,p_1\}, \{m_1, p_1\}^c\notin{F}$. Then at least one of $m_1$ or $p_1$ is removable.
\end{lemma}

\begin{proof} First suppose there exists an $m_1$-dangerous set $U_1$, i.e. $U_1^{c-}=\{m_1\}$. Note that by \ref{AP1}, $p_1\notin{U_1}$. By \Cref{lem:m-m-p-p-set-set}, there is no $p_1$-dangerous set. Suppose, for the sake of contradiction, there exists a $p_1$-dangerous pair $U,W$, i.e., $U\cup{W}=V$ and $(U\cap{W})^+=\{p_1\}$. Hence we may assume $p_2\notin{U}$. Since $p_1\in{U}$ and by \ref{AP1}, we have $m_1\in{U}$. 
If $p_2\in{U_1}$, then $U_1^c=\{m_1, p_1\}$, a contradiction to the assumption that $\{m_1, p_1\}^c\notin{F}$. Hence $p_2\notin{U_1}$. As $p_2\notin{U_1\cup{U}}$, we have $U\cup{U_1}\neq{V}$, however since all the $-$ signed vertices are covered by either $U$ or $U_1$, we have $(U\cup{U_1})^c\subset{V^+}$, a contradiction to \Cref{rem:minus-out-union}. Therefore, we cannot have a $p_1$-dangerous pair in $F$. 
This shows that if there exist a $m_1$-dangerous set, then $p_1$ is removable.

Now assume there is no $m_1$-dangerous set but we have a $m_1$-dangerous pair $W_1, W_2$, i.e. $W_1\cap{W_2}=\emptyset$ and $(W_1\cup{W_2})^{c-}=\{m_1\}$. Then we may assume $m_2\in{W_1}$ and by \ref{AP1}, we have $p_1\notin{W_1, W_2}$. If $p_2\in{W_1}$, since $W_1$ is an interval on the circle not including $\{m_1, p_1\}$, we have $W_1=V\backslash\{m_1, p_1\}$ and since $m_1, p_1\notin{W_2}$, we get $W_2\subseteq{W_1}$, a contradiction to $W_1\cap{W_2}=\emptyset$. Hence $p_2\notin{W_1}$. 

By \Cref{lem:m-m-p-p-pair-pair}, there is no $p_1$-dangerous pair. Now for the sake of contradiction, assume there exist a $p_1$-dangerous set $U_2$, i.e. $U_2^+=\{p_1\}$. By \ref{AP1}, $m_1\in{U_2}$. Since $\{m_1, p_1\}\notin{F}$, $m_2\in{U_2}$. 
Since $m_2\in{W_1}$, we have $U_2\cap{W_1}\neq{\emptyset}$. Hence by \Cref{rem:plus-in-intersection}, there exist a $+$ signed vertex $v\in{U_2\cap{W_1}}$. Since $p_1\notin{W_1}$, we have found a $+$ signed vertex $v\neq{p_1}$ in $U_2$, a contradiction. 

We proved that with the conditions of the lemma, if there exist a $m_1$-dangerous set or pair, then there is no $p_1$-dangerous set or pair and hence $p_1$ is removable. Furthermore, if there is no $m_1$-dangerous set or pair, we can remove $m_1$. This completes the proof. 
\end{proof}

\begin{lemma}\label{lem:m-m-p-p-2-block}
   Let $(V,\sigma,F)$ be an instance of the circle problem with additional properties \ref{AP1} and \ref{AP2}. 
   Suppose four consecutive vertices $m_2,m_1, p_1, p_2$ have signs $-, -, +, +$ respectively. Suppose $\{m_1, p_1\}\in{F}$ and $m_1$ is not removable. Let $p_1, p_2, ..., p_t$ be the vertices of the sign-block of $p_1$ in this order for some $t\ge 2$. For any $i\in[t]$, let $A_i=\{(m_2, p_s)|s\in\{1,2,..., i\}\}\cup\{(m_1, p_1)\}$ and $V_i=V\backslash\{m_1, p_1, p_2, ..., p_{i-1}\}$ and $\sigma_i=\sigma[V_i]$. Then there exist a $j\in\{1, 2, ..., t\}$ such that $(V_j,\sigma_j, (F\backslash{F_{A_j}})[V_j])$ is an instance of the circle problem and if $E_j$ is a solution to $(V_j,\sigma_j, (F\backslash{F_{A_j}})[V_j])$, then $E_j\cup{A_j}$ is a solution to $(V,\sigma,F)$. 
\end{lemma}

\begin{proof}
Let $U^*=\{m_1, p_1\}$. Since $U^*$ is a $p_1$-dangerous set, by \Cref{lem:m-m-p-p-set-set}, there is no $m_1$-dangerous set. Therefore, since $m_1$ is not removable, there exist a pair of sets $U,W\in{F}$ such that $U\cap{W}=\emptyset$ and $(U\cup{W})^{c-}=\{m_1\}$. By \ref{AP1}, $p_1\notin{W\cup{U}}$. 
For any such $m_1$-dangerous pair $(U,W)$, exactly one of them does not contain $m_2$, say $W$. Since $p_1\notin{W}$ and $(U\cup{W})^{c-}=\{m_1\}$, it follows from \ref{AP2} that there exist an $j\in\{2, 3, ..., t\}$ such that $p_j\in{W}$ and $p_{j-1}\notin{W}$. Denote as $f(U,W)=j$. W.l.o.g we assume $(U,W)$ has the maximum $f$ among all $m_1$-dangerous pairs. Consider $A=A_j$, $V'=V\backslash\{m_1, p_1, p_2, ..., p_{j-1}\}$, $\sigma'=\sigma[V']$ and $F'=(F\backslash{F_A})[V']$. We first show that $(V',\sigma', F')$ is an instance of the circle problem. 

As $m_2\in{U}$, $m_1\notin{U}$, $U\cap{W}=\emptyset$, and $U$ is an interval on the circle, we have $p_1, p_2, ..., p_{j-1}, p_j\notin{U}$. 
Figure~\ref{fig:alg-steps-combined}.c shows the sets $U^*, U, W$ around the circle and $A$. Using \Cref{lem:P0-P1-remains}, properties \ref{P0} and \ref{P1} hold for $(V',\sigma',F')$. 

To show \ref{P2} for $F'$, let $C\in{F}$ that is not covered by $A$ and let $C'=C[V']\in{F'}$. If $C'^+=\emptyset$, then $C^+\subseteq\{p_1, p_2, ..., p_{j-1}\}$ and is $C^+$ non-empty by \ref{P2} for $F$. Since $C$ is not covered by $A$, $m_2\in{C}$. Since $m_2\in{U}$, then $U\cap{C}\neq{\emptyset}$. However since $C^+\subseteq\{p_1, p_2, ..., p_{j-1}\}$ and ${p_1, p_2, ..., p_{j-1}}\notin{U}$, we get $(U\cap{C})^+=\emptyset$, a contradiction to \Cref{rem:plus-in-intersection}. This shows $C'^+\neq{\emptyset}$.\\
Now suppose $C'^-=\emptyset$. Then $C^{c-}=\{m_1\}$. By \ref{AP1}, $p_1\notin{C}$. Hence $U^*\cap{C}=\emptyset$. However, $U^*\cup{C}$ includes all $V^-$, a contradiction to \ref{P3}. This means $C'^-\neq\emptyset$. Thus \ref{P2} holds for $F'$. 

Before showing \ref{P3} and \ref{P4}, we prove that for any $m_1$-dangerous pair $C,D\in{F}$, at least one of them is covered by $A$. From the pair let $D$ be the set such that $m_2\notin{D}$. By the choice of $(U,W)$, $f(C,D)\leq{f(U,W)}$. Hence there exists an $i\in\{2,3, ..., j\}$ for $j=f(U,W)$ such that $p_i\in{D}$. Hence $D$ is covered by the arc $(m_2, p_i)$. 

To show \ref{P3} for $F'$, let $C,D\in{F}$ that are not covered by $A$ such that $C'=C[V']\in{F'}, D'=D[V']\in{F'}$ and $C'\cap{D'}=\emptyset$. Suppose for the sake of contradiction, $(C'\cup{D'})^{c-}=\emptyset$. If $C\cap{D}=\emptyset$, then we must have $(C\cup{D})^{c-}=\{m_1\}$ by \ref{P3} for $F$. Hence $(C,D)$ is a $m_1$-dangerous pair and so one of them is covered by $A$, a contradiction. Therefore, $C\cap{D}\neq{\emptyset}$. Since $C'\cap{D'}=\emptyset$, we must have $C\cap{D}\subseteq\{m_1, p_1, p_2, ..., p_{j-1}\}$. Hence we may assume $m_2\notin{C}$ and by \Cref{rem:plus-in-intersection}, there is a $p_i\in{C\cap{D}}$ for $i\in[j-1]$. Hence $C$ is covered by the arc $(m_2, p_i)\in{A}$, a contradiction. This concludes \ref{P3} for $F'$. 

To show \ref{P4} for $F'$, let $C,D\in{F}$ that are not covered by $A$ such that $C'=C[V']\in{F'}, D'=D[V']\in{F'}$ and $C'\cup{D'}=V'$. Suppose for the sake of contradiction, $(C'\cap{D'})^+=\emptyset$. 
\\
First suppose $C\cup{D}\neq{V}$. Since $C'\cup{D'}=V'$, we must have $(C\cup{D})^c\subseteq\{m_1, p_1, p_2, ..., p_{j-1}\}$. Using \Cref{rem:minus-out-union}, 
$(C\cup{D})^{c-}\neq{\emptyset}$ so $m_1\notin{C\cup{D}}$. Thus, by \ref{AP1}, $p_1\notin{C\cup{D}}$. If $C\cap{D}\neq{\emptyset}$, then by \ref{P1},  $C\cup{D}\in{F}$ and $((C\cup{D})\cup{U^*})^c\subset{V^+}$, a contradiction to \ref{P3}. Then $C\cap{D}=\emptyset$. However that means $(C,D)$ is an $m_1$-dangerous pair and so one of them is covered by $A$. Therefore $C\cup{D}=V$. \\
By \ref{P4} for $F$, $(C\cap{D})^+\neq{\emptyset}$ and since $(C'\cap{D'})^+=\emptyset$, we have $(C\cap{D})^+\subseteq\{p_1, p_2, ..., p_{j-1}\}$. Since $C,D$ are not covered by $A$, we have $m_2\in{C,D}$. We also may assume $p_j\notin{C}$. Then since $C$ is an interval and $C\cap\{p_1, p_2, ..., p_{j-1}\}\neq\emptyset$, there exist some $k\in[j-1]$ such that $\{m_2, m_1\}\cup\{p_1, p_2, ..., p_k\}\subseteq{C}$ and $p_{k+1}\notin{C}$. 
Since $p_j\notin{U}$, then $C\cup{U}\neq{V}$ and since $m_2\in{U}$, we have $C\cap{U}\neq{\emptyset}$. Then since $F$ is crossing, we have $C\cup{U}\in{F}$. We show that $U$ is not a subset of $C$ and therefore $C\cup{U}=U\cup\{m_1, p_1, p_2, ..., p_k\}$. 
This is because $m_2\in{D\cap{U}}$ and hence $D\cap{U}\neq{\emptyset}$ and by \Cref{rem:plus-in-intersection}, there exists a $+$ signed vertex $p\in{D\cap{U}}$ where $p\notin\{p_1, p_2, ..., p_j\}$ as $p_1, p_2, ..., p_j\notin{U}$, and since $(C\cap{D})^+\subseteq\{p_1, p_2, ..., p_{j-1}\}$, we have $p\notin{C}$. This shows that $U$ is not a subset of $C$. \\
We have $C\cup{U}=U\cup\{m_1, p_1, p_2, ..., p_k\}$. Observe that $W\cap{(C\cup{U})}=\emptyset$ and we have $(W\cup{(C\cup{U})})^c\subseteq{(W\cup{U}\cup\{m_1, p_1\})^c}\subset{V^+}$, a contradiction to \ref{P3} for $F$. This concludes proof of \ref{P4} for $F'$. 

We showed that $(V',\sigma',F')$ is an instance of the circle problem. 
Then by construction if $E'$ is a solution to $(V',\sigma', F')$, 
$E'\cup{A}$ is outer-planar. For every set $C\in{F}$ such that $C\subseteq{V\backslash{V'}}$, $C$ is covered by arc $(m_2, p_i)\in{A}$ for some $i\in[j-1]$ that $p_i\in{C}$. Hence by \Cref{lem:union-covering}, $E'\cup{A}$ covers $F$. Therefore,  
$E'\cup{A}$ will be a solution to $(V,\sigma,F)$. 
\end{proof}

\begin{lemma}\label{lem:m-m-p-p-co-2-block}
Let $(V,\sigma,F)$ be an instance of the problem with additional properties \ref{AP1} and \ref{AP2}. Suppose four consecutive vertices $m_2,m_1,p_1,p_2$ have signs $-, -, +, +$ respectively. Suppose $\{m_1, p_1\}^c\in{F}$ and $p_1$ is not removable. Let $m_\ell, m_{\ell-1}, ..., m_2, m_1$ be the consecutive vertices of sign-block of $m_1$ for some $\ell\ge 2$. For any $i\in{\ell}$, let $A_i=\{(m_s, p_2)|s\in\{1, 2, ..., i\}\}\cup\{(m_1, p_1)\}$ and $V_i=V\backslash\{p_1, m_1, m_2, ..., m_{i-1}\}$ and $\sigma_i=\sigma[V_i]$. Then there exist a $j\in{1, 2, ..., \ell}$ such that $(V_j,\sigma_j, (F\backslash{F_{A_j}})[V_j])$ is an instance of the circle problem and if $E_j$ is a solution to $(V_j,\sigma_j, (F\backslash{F_{A_j}})[V_j])$, then $E_j\cup{A_j}$ is a solution to $(V,\sigma,F)$. 
\end{lemma}

\begin{proof}
    This is equivalent to \Cref{lem:m-m-p-p-2-block} for the dual instance. See figure~\ref{fig:alg-steps-combined}.d.
\end{proof}

\begin{lemma}\label{lem:no-m-m-p-p}
Let $(V,\sigma,F)$ be an instance of the circle problem with additional properties \ref{AP1} and \ref{AP2}. Then we may assume there is no four consecutive vertices $m_2, m_1, p_1, p_2$ around the circle where $m_1, m_2$ have a $-$ sign and $p_1, p_2$ have a $+$ sign. 
\end{lemma}

\begin{proof}
This is a consequence of \Cref{lem:m-m-p-p-simple} when $\{m_1,p_1\},\{m_1,p_1\}^c\notin{F}$, of \Cref{lem:m-m-p-p-2-block} if $\{m_1,p_1\}\in{F}$ and \Cref{lem:m-m-p-p-co-2-block} otherwise. 
\end{proof}

\begin{lemma}\label{lem:m-m-p-m}
Let $(V,\sigma,F)$ be an instance of the circle problem with additional properties \ref{AP1} and \ref{AP2}. Suppose four consecutive vertices $m_2,m_1,p,m_3$ have signs $-,-,+,- $ respectively. Then $m_1$ is removable. 
\end{lemma}
\begin{proof}
    Suppose otherwise. Let $U_1$ be an $m_1$-dangerous set, i.e. $U_1^{c-}=\{m_1\}$. Then $m_3\in{U_1}$ and by \ref{AP1}, $p\notin{U_1}$. However this is a contradiction to \ref{AP2}. Now suppose there exist a $m_1$-dangerous pair $(U,W)$. Then we may assume $m_3\in{W}$. Similarly by \ref{AP1}, we have $p\notin{W}$ and together with $m_3\in{W}$, we get a contradiction to \ref{AP2}. 
\end{proof}

\begin{lemma}\label{lem:p-p-m-p}
Let $(V,\sigma,F)$ be an instance of the circle problem with additional properties \ref{AP1} and \ref{AP2}. 
Suppose four consecutive vertices $p_2,p_1,m,p_3$ have signs $+,+,-,+$ respectively. Then $p_1$ is removable. 
\end{lemma}
\begin{proof}
    This is equivalent to \Cref{lem:m-m-p-m} for the dual instance.
\end{proof}

\subsection{Proof of \Cref{thm:circle}}\label{proof-of-circle-thm}
This is a proof by contradiction. Consider an instance $(V,\sigma, F)$ of the circle problem with the smallest value of $|V|+|F|$ such that there is no outer-planar subset of arcs from $V^-$ to $V^+$ that covers $F$. Note that $F\neq{\emptyset}$. By Lemmas~\ref{lem:plus-end-plus-out} and \ref{lem:ends-minus-minus-out}, we may assume \ref{AP1} and \ref{AP2} hold for $F$. By Lemmas~\ref{lem:no-m-m-p-p}, \ref{lem:m-m-p-m} and \ref{lem:p-p-m-p}, the signs of vertices of $V$ are alternating around the circle. Then by properties \ref{AP1} and \ref{AP2}, $F=\emptyset$, a contradiction. Therefore, \Cref{thm:circle} holds for all instances of the circle problem.

\section{Bridge Lemma}\label{sec:bridge-lemma-with-proof}
In this section we formally state the Bridge Lemma that connects \Cref{thm:circle} and \Cref{thm:main}.

Let $(D=(V,A), w)$ be a $0,1$-weighted digraph. For any vertex $v\in{V}$, denote by $N_D^-(v)$ or simply $N^-(v)$ the set of vertices that have an outgoing arc to $v$. Furthermore, denote by $N_D^+(v)$ or simply $N^+(v)$ the set of vertices that get an incoming arc from $v$. A vertex $u$ is a source if $|N^-(u)|=0$ and $|N^+(u)|\ge 1$. Furthermore, a vertex $v$ is a \textit{near-source} if $|N^-(v)|=1$ and $|N^+(v)|\ge 1$. Analogously, a vertex $w$ is a sink if $|N^+(w)|=0$ and $|N^-(w)|\ge 1$ and is a \textit{near-sink} if $|N^+(w)|=1$ and $|N^-(w)|\ge 1$. 
We say that $v$ is a \textit{weight-$0$} vertex if all the arcs incident to $v$ have weight $0$. Denote by $D\backslash\{v\}$ the graph obtained by removing vertex $v$ with all its edges from $D$. 

Let $\delta^+(U)$ be a dicut in $D$. We say that $\delta^+(U)$ is \textit{trivial} if either $|U|=1$ or $|V\backslash{U}|=1$, i.e. $U$ only consist of a source or $V\backslash{U}$ only consist of a sink.

Denote by a \textit{proper} $0,1$-weighted plane digraph $(D,w)$, when every dicut has weight at least $2$ and the weight-1 arcs form a connected component, not necessarily spanning.

We say that a $0,1$-weighted digraph $(D,w)$ is \textit{super-proper} if it is proper and has the following properties:

\begin{enumerate}[label={\bfseries A\arabic*}, start=1]
    \item\label{A1} $D$ is acyclic.
    \item\label{A2} Every minimum weight dicut in $D$ is trivial.
    \item\label{A3} If $x$ is a weight-$0$ vertex in $D$, then $x$ is not a cut vertex.
    \item\label{A4} If $x$ is a weight-$0$ vertex in $D$, then $x$ is not in a $2$-cut-set with a source/sink $v$ of weighted degree $2$ where the weight-$1$ arcs of $v$ go to two different components of $D\backslash{\{x,v\}}$. 
    \item\label{A5} $D$ has no weight-$0$ near-source or near-sink vertex.
\end{enumerate}

We now state the Bridge Lemma.

\begin{lemma}\label{lem:bridge-lemma}
    Let $(D,w)$ be a super-proper $0,1$-weighted plane digraph. 
    Let $v$ be a weight-$0$ vertex. Let $N^-(v)=\{l_1, l_2, ..., l_k\}$ and $N^+(v)=\{r_1,r_2, ..., r_t\}$ for some $k,t\ge 1$. Let $A=\{(l_i, r_j)|\forall\; i\in[k], j\in[k]\}$ be a set of weight-$0$ arcs. Then there exists $A'\subseteq{A}$ such that the obtained weighted digraph $(D',w')$ by removing $v$ and adding $A'$ is a proper $0,1$-weighted plane digraph. Furthermore, if the weight-$1$ arcs can be decomposed into two dijoins $J_1, J_2$ of $D'$, then $J_1,J_2$ are also dijoins of $D$. 
\end{lemma}

\subsection{Proof of \Cref{lem:bridge-lemma} (Bridge Lemma)}
Let ${F}=\{U\in{V\backslash{\{v\}}}|\delta^+(U)\text{ is a dicut of weight at most 1 in }{D\backslash\{v\}}\}$. 
To obtain a proper $0,1$-weighted plane digraph from $D\backslash\{v\}$, we need to remove all dicuts in ${F}$. We do that by finding a set of arcs $A'\subset{A}$ such that every set in ${F}$ has an incoming arc from $A'$ while preserving the planarity. To do so, we show that $(N(v),\sigma, {F})$ is an instance of the circle problem where the circle $C$ is obtained by the circular ordering of $N(v)$ around $v$ in a plane drawing of $D$. The function $\sigma:N(v)\rightarrow{\{+,-\}}$ is defined as follows: For every vertex in $x\in{N(v)}$, $x$ gets a $-$ sign if $x\in{N^-(v)}$ and a $+$ sign if $x\in{N^+(v)}$. 

\begin{claim}\label{clm:no-weight-0-dicut-in-D-v}
$D\backslash\{v\}$ has no weight-$0$ dicut. 
\end{claim}
\begin{cproof}
    Suppose otherwise. Let $\delta^+(U)$ be a weight-$0$ dicut in $D\backslash\{v\}$ and let $\bar{U}=V\backslash{(\{v\}\cup{U})}$. Since the weight-$1$ arcs form one connected component then one of $(D\backslash\{v\})[U]$ or $(D\backslash\{v\})[\bar{U}]$ has only weight-$0$ arcs. W.l.o.g let that be $(D\backslash\{v\})[U]$. By property \ref{A1}, $D$ is acyclic, hence there exist a vertex $u$ that is a source in $(D\backslash\{v\})[U]$. $u$ has an incoming arc from $v$ in $D$, otherwise $\delta^+\{u\}$ would be a weight-$0$ dicut in $D$. However this means $u$ is a near-source in $D$, a contradiction. Hence $D\backslash\{v\}$ has no weight-$0$ dicut. 
\end{cproof}

\begin{claim}\label{clm:P1-holds}
    \ref{P1} holds for ${F}$, i.e, ${F}$ is a crossing family. 
\end{claim}
\begin{cproof}
    By Claim~\ref{clm:no-weight-0-dicut-in-D-v}, ${F}$ only contains dicuts of weight $1$ in $D\backslash\{v\}$ and minimum weight of a dicut in $D\backslash\{v\}$ is $1$. Let $U,W\in{{F}}$ such that $U\cap{W}\neq{\emptyset}$ and $U\cup{W}\neq{V}$. Then $\delta^+(U\cap{W})$ and $\delta^+(U\cup{W})$ are also dicuts. Since $w(\delta^+(U\cap{W}))+w(\delta^+(U\cup{W}))=w(\delta^+(U))+w(\delta^+(W))=2$ and minimum weight of a dicut in $D\backslash\{v\}$ is $1$, we have $w(\delta^+(U\cap{W}))=w(\delta^+(U\cup{W}))=1$, hence $U\cap{W}, U\cup{W}\in{{F}}$. Therefore, ${F}$ is a crossing family.
\end{cproof}

\begin{claim}\label{clm:weight-1-dicut-both-sides-connected} 
Let $\delta^+(U)$ be a weight-$1$ dicut in $D\backslash\{v\}$ for non-empty set $U\subset{V\backslash\{v\}}$. Then $D[U]$ and $D[\bar{U}]$ are both weakly connected where $\bar{U}=V\backslash(U\cup\{v\})$. 
\end{claim}

\begin{cproof}
    Suppose otherwise and by symmetry suppose $D[U]$ is not weakly connected. Since the weight-$1$ arcs form one connected component, there exist a component $U_1$ in $D[U]$ that only has weight-$0$ arcs. By property \ref{A1}, $D$ is acyclic, therefore there exist a vertex $u\in{U_1}$ that is a source in $D[U_1]$ and therefore in $D\backslash\{v\}$. $u$ has an incoming arc from $v$ in $D$, otherwise $\delta^+\{u\}$ would be a weight-$0$ dicut in $D$. However this means $u$ is a near-source in $D$, a contradiction. This shows that $D[U]$ and $D[\bar{U}]$ must both be weakly connected. 
\end{cproof}

Recall that $C$ is obtained by the circular ordering of $N(v)$ around $v$ in a plane drawing of $D$. 

\begin{claim}\label{clm:p0-holds}
\ref{P0} holds for $(N(v),\sigma, {F})$, i.e, for any weight-$1$ dicut $\delta^+(U)$ of $D\backslash\{v\}$ for some non-empty set $U\subset{V\backslash\{v\}}$, $U\cap{C}$ is one contiguous interval. 
\end{claim} 

\begin{cproof}
Suppose otherwise. Then there exist $4$ disjoint and non empty contiguous intervals $I_1, I_2, I_3, I_4\subset{C}$ that $I_1, I_3\subset{U}$ and $I_2\cap{U}= I_4\cap{U}=\emptyset$, see figure~\ref{fig:contguous-interval}. 
\begin{figure}
    \centering
    \tikzset{every picture/.style={line width=0.75pt}} 

\begin{tikzpicture}[x=0.75pt,y=0.75pt,yscale=-0.75,xscale=0.75]

\draw  [dash pattern={on 0.84pt off 2.51pt}] (196.09,144.95) .. controls (196.09,101.35) and (231.44,66) .. (275.05,66) .. controls (318.65,66) and (354,101.35) .. (354,144.95) .. controls (354,188.56) and (318.65,223.91) .. (275.05,223.91) .. controls (231.44,223.91) and (196.09,188.56) .. (196.09,144.95) -- cycle ;
\draw  [fill={rgb, 255:red, 0; green, 0; blue, 0 }  ,fill opacity=1 ] (272.05,144.95) .. controls (272.05,143.3) and (273.39,141.95) .. (275.05,141.95) .. controls (276.7,141.95) and (278.05,143.3) .. (278.05,144.95) .. controls (278.05,146.61) and (276.7,147.95) .. (275.05,147.95) .. controls (273.39,147.95) and (272.05,146.61) .. (272.05,144.95) -- cycle ;
\draw    (234,53.91) -- (318,53.91) ;
\draw    (234,53.91) -- (234,68.91) ;
\draw    (318,53.91) -- (318,68.91) ;
\draw  [fill={rgb, 255:red, 0; green, 0; blue, 0 }  ,fill opacity=1 ] (272.05,66) .. controls (272.05,64.34) and (273.39,63) .. (275.05,63) .. controls (276.7,63) and (278.05,64.34) .. (278.05,66) .. controls (278.05,67.66) and (276.7,69) .. (275.05,69) .. controls (273.39,69) and (272.05,67.66) .. (272.05,66) -- cycle ;
\draw    (370,106.91) -- (370,171.91) ;
\draw    (370,106.91) -- (352,106.91) ;
\draw    (370,171.91) -- (352,171.91) ;
\draw  [fill={rgb, 255:red, 0; green, 0; blue, 0 }  ,fill opacity=1 ] (351,141.95) .. controls (351,140.3) and (352.34,138.95) .. (354,138.95) .. controls (355.66,138.95) and (357,140.3) .. (357,141.95) .. controls (357,143.61) and (355.66,144.95) .. (354,144.95) .. controls (352.34,144.95) and (351,143.61) .. (351,141.95) -- cycle ;
\draw    (178,110.91) -- (178,175.91) ;
\draw    (196,110.91) -- (178,110.91) ;
\draw    (196,175.91) -- (178,175.91) ;
\draw  [fill={rgb, 255:red, 0; green, 0; blue, 0 }  ,fill opacity=1 ] (193.09,144.95) .. controls (193.09,143.3) and (194.44,141.95) .. (196.09,141.95) .. controls (197.75,141.95) and (199.09,143.3) .. (199.09,144.95) .. controls (199.09,146.61) and (197.75,147.95) .. (196.09,147.95) .. controls (194.44,147.95) and (193.09,146.61) .. (193.09,144.95) -- cycle ;
\draw    (233.05,240.91) -- (317.05,240.91) ;
\draw    (317.05,225.91) -- (317.05,240.91) ;
\draw    (233.05,225.91) -- (233.05,240.91) ;
\draw  [fill={rgb, 255:red, 0; green, 0; blue, 0 }  ,fill opacity=1 ] (272.05,223.91) .. controls (272.05,222.25) and (273.39,220.91) .. (275.05,220.91) .. controls (276.7,220.91) and (278.05,222.25) .. (278.05,223.91) .. controls (278.05,225.56) and (276.7,226.91) .. (275.05,226.91) .. controls (273.39,226.91) and (272.05,225.56) .. (272.05,223.91) -- cycle ;
\draw  [draw opacity=0][dash pattern={on 4.5pt off 4.5pt}] (273.95,225.9) .. controls (264.61,228.82) and (254.61,230.41) .. (244.2,230.41) .. controls (192.85,230.41) and (151.22,191.8) .. (151.22,144.17) .. controls (151.22,96.55) and (192.85,57.94) .. (244.2,57.94) .. controls (254.61,57.94) and (264.61,59.53) .. (273.95,62.45) -- (244.2,144.17) -- cycle ; \draw  [dash pattern={on 4.5pt off 4.5pt}] (273.95,225.9) .. controls (264.61,228.82) and (254.61,230.41) .. (244.2,230.41) .. controls (192.85,230.41) and (151.22,191.8) .. (151.22,144.17) .. controls (151.22,96.55) and (192.85,57.94) .. (244.2,57.94) .. controls (254.61,57.94) and (264.61,59.53) .. (273.95,62.45) ;  
\draw  [draw opacity=0][dash pattern={on 4.5pt off 4.5pt}] (354.61,144.75) .. controls (358.73,156.34) and (361,169.03) .. (361,182.33) .. controls (361,237.05) and (322.56,281.41) .. (275.14,281.41) .. controls (227.72,281.41) and (189.27,237.05) .. (189.27,182.33) .. controls (189.27,169.86) and (191.27,157.92) .. (194.92,146.92) -- (275.14,182.33) -- cycle ; \draw  [dash pattern={on 4.5pt off 4.5pt}] (354.61,144.75) .. controls (358.73,156.34) and (361,169.03) .. (361,182.33) .. controls (361,237.05) and (322.56,281.41) .. (275.14,281.41) .. controls (227.72,281.41) and (189.27,237.05) .. (189.27,182.33) .. controls (189.27,169.86) and (191.27,157.92) .. (194.92,146.92) ;  

\draw (269,132) node [anchor=north west][inner sep=0.75pt]  [font=\scriptsize] [align=left] {$\displaystyle v$};
\draw (268,32) node [anchor=north west][inner sep=0.75pt]  [font=\scriptsize] [align=left] {$\displaystyle I_{1}$};
\draw (270.05,70) node [anchor=north west][inner sep=0.75pt]  [font=\scriptsize] [align=left] {$\displaystyle i_{1}$};
\draw (373,132) node [anchor=north west][inner sep=0.75pt]  [font=\scriptsize] [align=left] {$\displaystyle I_{2}$};
\draw (335.05,134) node [anchor=north west][inner sep=0.75pt]  [font=\scriptsize] [align=left] {$\displaystyle i_{2}$};
\draw (160,134) node [anchor=north west][inner sep=0.75pt]  [font=\scriptsize] [align=left] {$\displaystyle I_{4}$};
\draw (201,138) node [anchor=north west][inner sep=0.75pt]  [font=\scriptsize] [align=left] {$\displaystyle i_{4}$};
\draw (268.05,205) node [anchor=north west][inner sep=0.75pt]  [font=\scriptsize] [align=left] {$\displaystyle i_{3}$};
\draw (269,242) node [anchor=north west][inner sep=0.75pt]  [font=\scriptsize] [align=left] {$\displaystyle I_{3}$};
\draw (152,76) node [anchor=north west][inner sep=0.75pt]  [font=\scriptsize] [align=left] {$\displaystyle P_{1}$};
\draw (214,270) node [anchor=north west][inner sep=0.75pt]  [font=\scriptsize] [align=left] {$\displaystyle P_{2}$};
\draw (321,92) node [anchor=north west][inner sep=0.75pt]  [font=\scriptsize] [align=left] {$\displaystyle C$};

\end{tikzpicture}
    \caption[contiguous-interval]{For Claim~\ref{clm:p0-holds}}
    \label{fig:contguous-interval}
    \end{figure}
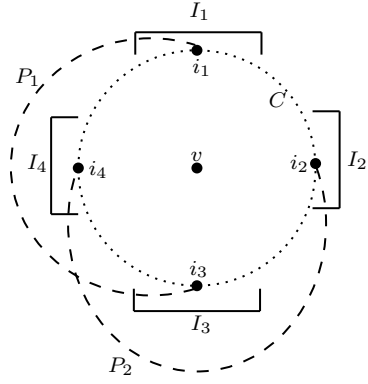
By Claim~\ref{clm:weight-1-dicut-both-sides-connected}, $D[U]$ and $D[V\backslash(U\cup\{v\})]$ are weakly connected. Therefore there exist a path $P_1$ between a vertex $i_1\in{I_1}$ and a vertex $i_3\in{I_3}$ where vertices of $P_1$ are in $U$. Moreover, there exist a path $P_2$ between a vertex $i_2\in{I_2}$ and a vertex $i_4\in{I_4}$ where vertices of $P_2$ are in $V\backslash(U\cup\{v\})$. Since $N(v)$ are around a circle $C$ with center $v$, thus both paths $P_1, P_2$ are outside of $C$ because they don't use $v$. Therefore paths $P_1$ and $P_2$ must intersect, a contradiction to planarity. This proves that $U\cap{C}$ must be one contiguous interval. 
\end{cproof}

\begin{claim}\label{clm:p2-holds}
\ref{P2} holds for $(N(v),\sigma, {F})$, i.e, for any weight-$1$ dicut $\delta^+(U)$ of $D\backslash\{v\}$ for some non-empty set $U\subset{V\backslash\{v\}}$, there exist vertices $m\in{N_D^-(v)}$ and $p\in{N_D^+(v)}$ such that $m\notin{U}$ and $p\in{U}$. 
\end{claim}
\begin{cproof}
   Suppose otherwise. If $N^+(v)\cap{U}=\emptyset$, then $\delta^+(U)$ is also a weight-$1$ dicut in $D$, a contradiction. Moreover if $N^-(v)\cap{(V\backslash(U\cup\{v\}))}=\emptyset$, then $\delta^+(U\cap\{v\})$ is a weight-$1$ dicut in $D$, a contradiction. 
\end{cproof}

\begin{claim}\label{clm:p3-holds}
\ref{P3} holds for $(N(v),\sigma,{F})$, i.e, for any two weight-$1$ dicuts $\delta^+(U), \delta^+(W)$ in $D\backslash\{v\}$ for some non-empty subsets $U, W\subset{V\backslash\{v\}}$ such that $U\cap{W}=\emptyset$, there exists a vertex $m\in{N_D^-(v)}$ such that $m\notin{U\cup{W}}$. 
\end{claim}

\begin{cproof}
    Note that $\delta^+(U\cup{W})$ is a dicut of weight $2$ in $D\backslash\{v\}$. For the sake of contradiction, suppose $(V\backslash(U\cup{W})\cap{N_D^-(v)}=\emptyset$. Then $\delta^+(U\cup{W}\cup\{v\})$ is a dicut of weight $2$ in $D$ and by property \ref{A2}, it is trivial. Hence $V\backslash(U\cup{W}\cup\{v\})$ is a sink vertex $u$ in $D$ with weighted degree $2$ that gets one weight-$1$ arc from each of $U$ and $W$. Moreover, $(u,v)$ is a $2$-cut-set, a contradiction to property \ref{A4}. 
\end{cproof}

\begin{claim}\label{clm:p4-holds}
\ref{P4} holds for $(N(v),\sigma, {F})$, i.e, for any two weight-$1$ dicuts $\delta^+(U), \delta^+(W)$ in $D\backslash\{v\}$ for non-empty subsets $U, W\subset{V\backslash\{v\}}$ such that $U\cup{W}=V\backslash\{v\}$, there exists a vertex $p\in{N_D^+(v)}$ such that $p\in{U\cap{W}}$. 
\end{claim}
\begin{cproof}
    Note that $\delta^+(U\cap{W})$ is a dicut of weight $2$ in $D\backslash\{v\}$. For the sake of contradiction, suppose $(U\cap{W})\cap{N_D^+(v)}=\emptyset$. Then $\delta^+(U\cap{W})$ is also a dicut of weight $2$ in $D$ and by property~\ref{A2}, it is trivial. Hence $U\cap{W}$ is a source vertex $u$ in $D$ with weighted degree $2$ that gives one gives one weight-$1$ arc to each of $U\backslash\{u\}$ and $W\backslash\{u\}$. Moreover, $(u,v)$ is a $2$-cut-set, a contradiction to property~\ref{A4}. 
\end{cproof}
By the above claims, we showed that the properties \ref{P0}-\ref{P4} hold for $(N(v),\sigma, {F})$. Hence $(N(v),\sigma, {F})$ is is an instance of the circle problem. Therefore by \Cref{thm:circle}, there exist an outer-planer set of arcs $A'$ from $N^-(v)$ to $N^+(v)$ such that every set in ${F}$ has an incoming arc. Let $(D',w')$ be the digraph obtained by removing $v$ and adding $A'$. By construction $D$ is planar and has no dicut of weight $0$ or $1$. Clearly the weight-$1$ arcs still form one connected component in $D'$. Hence, $(D',w')$ is a proper $0,1$-weighted plane digraph. Let $\delta^+(U)$ be a dicut of $D$. Then $\delta^+(U\backslash\{v\})$ is a dicut of $D'$. Therefore, if $J_1, J_2$ are disjoint dijoins of $D'$ among the weight-$1$ arcs, they are also dijoins of $D$. This completes the proof of \Cref{lem:bridge-lemma}.
\qed

\section{\Cref{thm:main}}\label{sec:dijointhmproof}
Note that \Cref{thm:main} states that the weight-$1$ arcs of a proper $0,1$-weighted plane digraph $(D,w)$ can be decomposed into two dijoins of $D$. The crux of its proof is to eliminate the weight-$0$ vertices and replacing them by a plane gadget while preserving the minimum weight of a dicut. After that we use the following result from Chudnovsky et al. \cite{Chudnovsky16}.

\begin{theorem}[\cite{Chudnovsky16}]\label{thm:chudnovsky}
Let $(D,w)$ be a proper $0,1$-weighted plane digraph where weight-$1$ arcs form a spanning tree. Then they can be decomposed into two dijoins of $D$. 
\end{theorem}

We can extend this theorem as follows. The weight-$1$ arcs only need to be spanning and connected, not necessarily a tree. In another words, the set of undirected weight-$1$ arcs can contain cycles. 

\begin{corollary}\label{cor:circle-removal}
    \Cref{thm:chudnovsky} holds if the weight-$1$ arcs form a spanning connected set.
\end{corollary}
\begin{proof}
The following argument can be extended easily for any number of cycles in the underlying undirected subgraph of weight-$1$ arcs of $D$. We only include the proof when this subgraph contains one cycle $C$. Fix an orientation of $C$ as clock-wise, and let $C_1$ be the arcs of $C$ that are oriented clock-wise in $D$ and let $C_2=C\backslash{C_1}$. Let $(D',w')$ be obtained by contracting $C$. Since we are contracting a connected component, $D'$ is planar and its weight-$1$ arcs remain a connected component, in this case a tree. Also, contraction does not decrease the minimum weight of a dicut, hence $(D',w')$ is proper. Then by \Cref{thm:chudnovsky}, weight-$1$ arcs of $D'$ can be decomposed into two dijoins $J_1',J_2'$ of $D'$. \\
    We show that $J_1=J_1'\cup{C_1}$ and $J_2=J_2'\cup{C_2}$ are dijoins of $D$. 
    Let $\delta^+(X)$ be a dicut in $D$. If it does not separate $C$, $\delta^+(X)$ remains a dicut in $D'$ and therefore intersects both $J_1',J_2'$. Otherwise $\delta^+(X)$ separates $C$. Note that $C_1,C_2\neq{\emptyset}$, otherwise no dicut would separate $C$. $\delta^+(X)$ must intersect the cycle $C$ at least twice and since it's a dicut, $\delta^+(X)\cap{C_1}\neq{\emptyset}$ and $\delta^+(X)\cap{C_2}\neq{\emptyset}$. This shows that $J_1,J_2$ are dijoins of $D$.
\end{proof}
\subsection{Proof of \Cref{thm:main}}\label{sec:proof-main-thm}
First, we show we may assume the digraph of \Cref{thm:main} is super-proper. In each claim, we reduce to smaller instances, proving that disjoint dijoins of the reduced digraph are also dijoins of the original. The first two claims are routine.  

\begin{claim}\label{clm:acyclic-whole-graph}
     We may assume $D$ is acyclic, i.e, \ref{A1} holds for $D$. 
\end{claim}
\begin{cproof}
    Suppose $D$ has a directed cycle $C$.  Let $(D',w')$ be obtained by contracting $C$.
    Then dicuts of $D'$ are precisely dicuts of $D$.  
    Moreover, we are contracting a connected subgraph of $D$, therefore $D'$ is planar and the its weight-$1$ arcs form one connected component. Hence, $(D',w')$ is a proper $0,1$-weighted plane digraph. Since dicuts of $D'$ and $D$ are the same, if $J_1,J_2$ are $2$ disjoint dijoins of $D'$ among its weight-$1$ arcs, they are also dijoins of $D$. 
\end{cproof}

\begin{claim}\label{clm:trivial-min-weight-dicuts}
Let $\delta^+(U)$ be a minimum weight dicut of $D$. Then we may assume $\delta^+(U)$ is trivial, i.e, \ref{A2} holds for $D$. 
\end{claim}
\begin{cproof}
    Let $\delta^+(U)$ be a non-trivial minimum weight dicut of $D$. We may assume $\delta^+(U)$ is minimal, meaning there is no other dicut $\delta^+(U')\subset{\delta^+(U)}$. \\
    First we show $D[U]$ and $D[V\backslash{U}]$ are weakly connected. Suppose otherwise. Let $S$ be a proper subset of $U$ such that there is no edge between $S$ and $U\backslash{S}$. Then $\delta^+(S)$ is also a dicut and a subset of $\delta^+(U)$, a contradiction to the minimality assumption of $\delta^+(U)$. Similarly, if $T$ is a proper subset of $\bar{U}=V\backslash{U}$ such that there is no edge between $T$ and $\bar{U}\backslash{T}$, then $\delta^+(U\cup{T})$ is also a dicut and a subset of $\delta^+(U)$, a contradiction to the minimality assumption of $\delta^+(U)$. Therefore $D[U]$ and $D[\bar{U}]$ are both weakly connected. 
    \begin{figure}
    \centering
    \tikzset{every picture/.style={line width=0.75pt}} 

\begin{tikzpicture}[x=0.65pt,y=0.65pt,yscale=-0.8,xscale=0.8]

\draw   (158,135.82) .. controls (172.91,135.82) and (185,159.14) .. (185,187.91) .. controls (185,216.68) and (172.91,240) .. (158,240) .. controls (143.09,240) and (131,216.68) .. (131,187.91) .. controls (131,159.14) and (143.09,135.82) .. (158,135.82) -- cycle ;
\draw  [dash pattern={on 4.5pt off 4.5pt}]  (161,156.82) -- (249,155.32) ;
\draw  [dash pattern={on 4.5pt off 4.5pt}]  (162,174.82) -- (251,174.32) ;
\draw [color={rgb, 255:red, 0; green, 0; blue, 0 }  ,draw opacity=1 ]   (162,193.82) -- (251,193.32) ;
\draw [color={rgb, 255:red, 0; green, 0; blue, 0 }  ,draw opacity=1 ]   (163,209.82) -- (251,209.32) ;
\draw  [dash pattern={on 4.5pt off 4.5pt}]  (163,225.82) -- (252,225.32) ;
\draw   (199,151) -- (209,156.41) -- (199,161.82) ;
\draw   (198,169) -- (208,174.41) -- (198,179.82) ;
\draw   (197,220) -- (207,225.41) -- (197,230.82) ;
\draw  [color={rgb, 255:red, 0; green, 0; blue, 0 }  ,draw opacity=1 ] (198,188) -- (208,193.41) -- (198,198.82) ;
\draw  [color={rgb, 255:red, 0; green, 0; blue, 0 }  ,draw opacity=1 ] (197,204) -- (207,209.41) -- (197,214.82) ;
\draw   (257,135.82) .. controls (271.91,135.82) and (284,159.14) .. (284,187.91) .. controls (284,216.68) and (271.91,240) .. (257,240) .. controls (242.09,240) and (230,216.68) .. (230,187.91) .. controls (230,159.14) and (242.09,135.82) .. (257,135.82) -- cycle ;
\draw  [dash pattern={on 4.5pt off 4.5pt}]  (362.5,126.5) -- (465,87.32) ;
\draw  [dash pattern={on 4.5pt off 4.5pt}]  (366,126.5) -- (468,108.82) ;
\draw [color={rgb, 255:red, 0; green, 0; blue, 0 }  ,draw opacity=1 ]   (366,126.5) -- (466,126.82) ;
\draw [color={rgb, 255:red, 0; green, 0; blue, 0 }  ,draw opacity=1 ]   (366,126.5) -- (466,145.82) ;
\draw  [dash pattern={on 4.5pt off 4.5pt}]  (366,126.5) -- (467,161.82) ;
\draw   (415.18,102.62) -- (422.41,103.41) -- (418.37,109.45) ;
\draw   (418.29,112.67) -- (425.95,115.8) -- (419.82,121.36) ;
\draw   (421.24,140.72) -- (426.63,148.1) -- (417.51,148.72) ;
\draw  [color={rgb, 255:red, 0; green, 0; blue, 0 }  ,draw opacity=1 ] (414,121) -- (421,125.91) -- (414,130.82) ;
\draw  [color={rgb, 255:red, 0; green, 0; blue, 0 }  ,draw opacity=1 ] (418.57,130.89) -- (424.22,137.28) -- (416.13,139.98) ;
\draw   (472,72.82) .. controls (486.91,72.82) and (499,96.14) .. (499,124.91) .. controls (499,153.68) and (486.91,177) .. (472,177) .. controls (457.09,177) and (445,153.68) .. (445,124.91) .. controls (445,96.14) and (457.09,72.82) .. (472,72.82) -- cycle ;
\draw   (377,207.82) .. controls (391.91,207.82) and (404,231.14) .. (404,259.91) .. controls (404,288.68) and (391.91,312) .. (377,312) .. controls (362.09,312) and (350,288.68) .. (350,259.91) .. controls (350,231.14) and (362.09,207.82) .. (377,207.82) -- cycle ;
\draw  [dash pattern={on 4.5pt off 4.5pt}]  (380,228.82) -- (470,265.32) ;
\draw  [dash pattern={on 4.5pt off 4.5pt}]  (381,246.82) -- (470,265.32) ;
\draw [color={rgb, 255:red, 0; green, 0; blue, 0 }  ,draw opacity=1 ]   (381,265.82) -- (470,265.32) ;
\draw [color={rgb, 255:red, 0; green, 0; blue, 0 }  ,draw opacity=1 ]   (382,281.82) -- (470,265.32) ;
\draw  [dash pattern={on 4.5pt off 4.5pt}]  (382,297.82) -- (470,265.32) ;
\draw   (424.38,242.08) -- (428.13,248.24) -- (421,249.32) ;
\draw   (420.5,250.83) -- (426.11,256.2) -- (419,259.32) ;
\draw   (418.98,280.85) -- (425.75,281.63) -- (422,287.32) ;
\draw  [color={rgb, 255:red, 0; green, 0; blue, 0 }  ,draw opacity=1 ] (421,261) -- (427,265.66) -- (421,270.32) ;
\draw  [color={rgb, 255:red, 0; green, 0; blue, 0 }  ,draw opacity=1 ] (417.9,272.11) -- (423.53,274.32) -- (419,278.32) ;
\draw  [fill={rgb, 255:red, 0; green, 0; blue, 0 }  ,fill opacity=1 ] (359,126.5) .. controls (359,124.57) and (360.57,123) .. (362.5,123) .. controls (364.43,123) and (366,124.57) .. (366,126.5) .. controls (366,128.43) and (364.43,130) .. (362.5,130) .. controls (360.57,130) and (359,128.43) .. (359,126.5) -- cycle ;
\draw  [fill={rgb, 255:red, 0; green, 0; blue, 0 }  ,fill opacity=1 ] (470,265.32) .. controls (470,263.39) and (471.57,261.82) .. (473.5,261.82) .. controls (475.43,261.82) and (477,263.39) .. (477,265.32) .. controls (477,267.25) and (475.43,268.82) .. (473.5,268.82) .. controls (471.57,268.82) and (470,267.25) .. (470,265.32) -- cycle ;
\draw   (298.53,161.64) -- (314.36,152.67) -- (312.45,149.3) -- (326.82,150.06) -- (320.09,162.77) -- (318.18,159.41) -- (302.35,168.38) -- cycle ;
\draw   (300.98,211.09) -- (317.11,219.5) -- (318.9,216.07) -- (326.07,228.54) -- (311.74,229.8) -- (313.53,226.37) -- (297.4,217.96) -- cycle ;

\draw (195,111) node [anchor=north west][inner sep=0.75pt]  [font=\footnotesize] [align=left] {$\displaystyle D$};
\draw (148,244) node [anchor=north west][inner sep=0.75pt]  [font=\footnotesize] [align=left] {$\displaystyle U$};
\draw (247,241) node [anchor=north west][inner sep=0.75pt]  [font=\footnotesize] [align=left] {$\displaystyle \overline{U}$};
\draw (463,177) node [anchor=north west][inner sep=0.75pt]  [font=\footnotesize] [align=left] {$\displaystyle \overline{U}$};
\draw (367,316) node [anchor=north west][inner sep=0.75pt]  [font=\footnotesize] [align=left] {$\displaystyle U$};
\draw (409,48) node [anchor=north west][inner sep=0.75pt]  [font=\footnotesize] [align=left] {$\displaystyle D_{1}$};
\draw (408,193) node [anchor=north west][inner sep=0.75pt]  [font=\footnotesize] [align=left] {$\displaystyle D_{2}$};
\draw (345,121) node [anchor=north west][inner sep=0.75pt]  [font=\footnotesize] [align=left] {$\displaystyle u$};
\draw (478,256) node [anchor=north west][inner sep=0.75pt]  [font=\footnotesize] [align=left] {$\displaystyle \overline{u}$};

\end{tikzpicture}
    \caption[non-trivial-dicut]{$D,D_1,D_2$ of Claim~\ref{clm:trivial-min-weight-dicuts}, solid arcs have weight $1$ and dashed arcs have weight $0$}
    \label{fig:non-trivial-dicut}
    \end{figure}
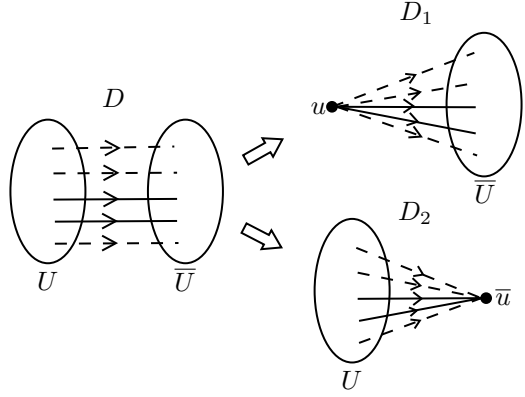
    
    Let $\{e,f\}$ be the two weight-$1$ arcs of $\delta^+(U)$. 
    Let $(D_1, w_1), (D_2, w_2)$ be the digraphs obtained by contracting $U, \bar{U}$ in $D$ to vertices $u,\bar{u}$ respectively, see figure~\ref{fig:non-trivial-dicut}. 
    Since $D[U]$ and $D[\bar{U}]$ are weakly connected, both $D_1,D_2$ are planar. Observe that weight-1 arcs in both $D_1,D_2$ form one connected component and minimum weight of a dicut is $2$. Hence 
    $(D_1, w_1), (D_2, w_2)$ are proper $0,1$-weighted planar digraphs. 
    
    Let $J_1, J_2$ be two disjoint dijoins among the weight-$1$ arcs for $D_1$ and let $J_3, J_4$ be two disjoint dijoins among the weight-$1$ arcs for $D_2$. We may assume $e\in{J_1,J_3}$ and $f\in{J_2,J_4}$. Let $J=J_1\cup{J_3}$ and $J'=J_2\cup{J_4}$. We show that $J,J'$ are dijoins for $D$. By symmetry, it is enough to show $J$ is a dijoin. \\
    Let $\delta^+(X)$ be a dicut of $D$. 
    For each case of $\delta^+(X)$, we go to a corresponding dicut in $D_1$ or $D_2$ and we show that $\delta^+(X)$ must intersect $J$. If $X\subset{U}$, then $\delta^+(X)$ remains a dicut in $D_2$ and must intersect $J_3$. Similarly if $X\subset{\bar{U}}$, then $\delta^+(X)$ remains a dicut in $D_1$ and must intersect $J_1$. Similar argument holds if $X$ contains one of $U$ or $\bar{U}$. Now suppose $X\cap{U}=Y$ and $X\cap{\bar{U}}=Y'$, both non-empty and proper. We have the following cases: 
    \begin{itemize}
        \item $e\in{\delta^+(X)}$, i.e. $e$ is from $Y$ to $\bar{U}\backslash{Y'}$, then $\delta^+(X)$ intersects $J$ in $e$. 
        \item $e\notin{\delta^+(X)}$ and is from $Y$ to $Y'$. Then $\delta^+(Y\cup\{u\})$ is a dicut in $D_1$ and there is an arc $h\in{\delta^+(Y\cup\{u\})\cap{J_1}}$ where $h$ is from $Y'$ to $\bar{U}\backslash{Y'}$. Hence $h\in{\delta^+(X)\cap{J}}$.  
        \item $e\notin{\delta^+(X)}$ and is from $U\backslash{Y}$ to $\bar{U}\backslash{Y'}$.  Then $\delta^+(Y)$ is a dicut in $D_2$ and there is an arc $g\in{\delta^+(Y)\cap{J_1}}$ where $g$ is from $Y$ to $U\backslash{Y}$. Hence $g\in{\delta^+(X)\cap{J}}$. 
    \end{itemize}
This shows that $J,J'$ are two disjoint dijoins among the weight-1 arcs for $D$. By this reduction we may assume that minimum weight dicuts of $D$ are trivial. 
\end{cproof}

\begin{claim}\label{clm:weight-0-cut-vertex}
Let $x$ be a weight-$0$ vertex in $D$. Then we may assume $x$ is not a cut vertex, i.e, \ref{A3} holds for $D$. 
\end{claim}
\begin{cproof}
    Let $x$ be a weight-$0$ vertex in $D$ and suppose $x$ is a cut vertex. Let $C_1,C_2,...C_k$ be the connected components of $D\backslash\{x\}$. Note that $x$ has an incoming and an outgoing arc to each $C_i$, otherwise we get a weight-0 dicut in $D$. Moreover note that all the weight-1 arcs should be in the same component, for example $C_1$, since they form one connected component. Let $(D_1, w_1)$ be the planar digraph obtained by contracting all the components $C_2,C_3,...,C_k$ in $D$ to vertices $u_2,u_3,...,u_k$ respectively. Observe that the weight-$1$ arcs remain one connected component in $D_1$. Since we are only contracting edges, the minimum weight of a dicut does not decrease in $D_1$. Hence $(D_1, w_1)$ is a proper $0,1$-weighted digraph. 
    
    Let $V_1$ be the vertices of the component $C_1$.
    Let $\delta^+(X)$ be a dicut of $D$. Note that $X\cap{V_1}\neq{\emptyset}$ since every dicut of $D$ has weight at least $2$. 
    If $x\in{X}$, let $X_1=(X\cap{V_1})\cup\{x, u_2, u_3,..., u_k\}$ and if $x\notin{X}$, let $X_1=X\cap{V_1}$. Then $\delta^+(X_1)$ is a dicut in $D_1$. Therefore, if $J_1, J_2$ are dijoins of $D_1$, they intersect $\delta^+(X_1)$ and hence $\delta^+(X)$ in their arcs from $X\cap{V_1}$ to $V_1\backslash{X}$, and therefore $J_1, J_2$ are also dijoins of $D$. By this reduction we may assume that the weight-$0$ vertices of $D$ are not cut vertices. 
\end{cproof}

\begin{claim}\label{clm:weight-0-two-cut-sets}
Let $x$ be a weight-$0$ vertex in $D$. Suppose $x$ is not a cut vertex. Then we may assume $x$ is not in a $2$-cut-set with a source/sink $v$ of weighted degree $2$ where the two weight-$1$ arcs of $v$ go to two different components of $D\backslash\{x,v\}$, i.e, \ref{A4} holds for $D$.  
\end{claim}
\begin{cproof}
    Suppose otherwise. By symmetry, we assume that $v$ is a source in $D$. Let $V$ be the vertex set of $D$. 
    First suppose that $D\backslash\{x,v\}$ has more than two components. Then there exist a component $C$ in $D\backslash\{x,v\}$ that only consists of weight-$0$ arcs. Note that if $C$ has only one vertex $c$, no dicut in $D$ separates $x$ and $c$, therefore we can identify them. Otherwise, let $(D',w')$ be obtained by contracting $C$ to a vertex $c$. It can be easily checked that $(D',w')$ is a proper $0,1$-weighted digraph. Let $\delta^+(X)$ be a dicut in $D$. If it does not separate $C$, it will remain a dicut in $D'$. Otherwise suppose $X\cap{C}=Y\neq{C, \emptyset}$. If $x\in{X}$, then $\delta^+(Y\cup{(V\backslash{C})})$ is a weight-$0$ dicut in $D$, a contradiction. Hence $x\notin{X}$. Then $\delta^+(Y')$ is a dicut in $D'$ where $Y'=X\cap{(V\backslash{C})}$. Therefore if $J_1, J_2$ are two disjoint dijoins in $D'$ among its weight-$1$ arcs, they are also dijoins of $D$. 

    Now suppose $D\backslash\{x,v\}$ has exactly two components $U, \bar{U}$. Let $e,f$ be the weight-1 arcs from $v$ to $U$ and from $v$ to $\bar{U}$ respectively. Note that x has both incoming and outgoing arcs to both $U$ and $\bar{U}$, otherwise we get a dicut of weight-1 in $D$. Let $(D_1, w_1),(D_2,w_2)$ be the digraphs obtained by contracting $U,\bar{U}$ in $D$ respectively to vertices $u_1,u_2$, as shown in figure~\ref{fig:two-cut-set}.  
    \begin{figure}[!hbt]
    \centering
    \input{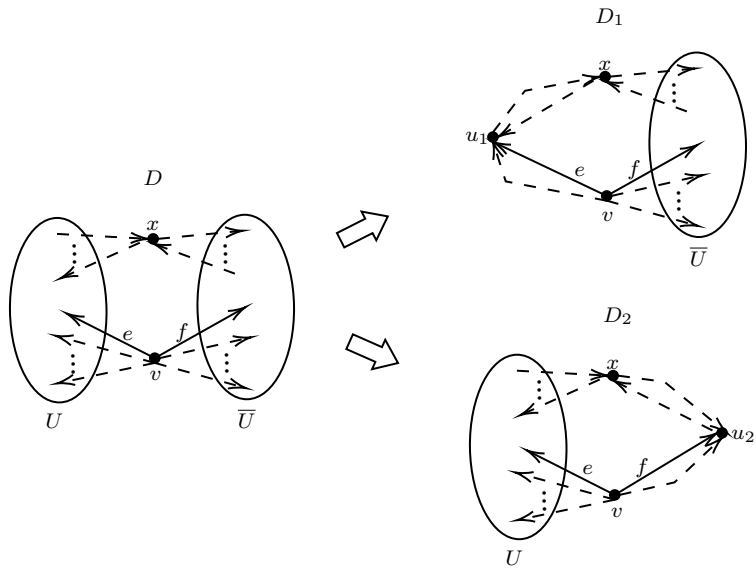}
    \caption[two-cut-set]{$D,D_1,D_2$ of Claim~\ref{clm:weight-0-two-cut-sets}, solid arcs have weight-$1$, and dashed arcs have weight-$0$}
    \label{fig:two-cut-set}
    \end{figure}
    It can be checked similarly to before that $(D_1,w_1), (D_2, w_2)$ are proper. Let $J_1, J_2$ be two disjoint dijoins of $D_1$ and let $J_3, J_4$ be two disjoint dijoins of $D_2$ among their weight-$1$ arcs. Since $\delta^+{v}$ is a minimum weight dicut, we may assume $e\in{J_1, J_3}$ and $f\in{J_2, J_4}$. Let $J=J_1\cup{J_3}$ and $J'=J_2\cup{J_4}$. We show that $J, J'$ are dijoins of $D$. By symmetry, it is enough to show that $J$ is a dijoin. 

    Let $\delta^+(X)$ be a dicut of $D$. For each case of $\delta^+(X)$, we go to a corresponding dicut in $D_1$ or $D_2$ and we show that $\delta^+(X)$ must intersect $J$. If $X\subset({U\cup\{x,v\})}$ or $X\subset({\bar{U}\cup\{x,v\})}$, then $\delta^+(X)$ remains a dicut in $D_2$ or $D_1$ respectively and hence it will intersect $J_3$ or $J_1$. Now let $X\cap{U}=Y\neq{U, \emptyset}$ and $X\cap{\bar{U}}=\bar{Y}\neq{\bar{U}, \emptyset}$. We have the following cases:
    \begin{itemize}
        \item $v\in{X}$: If $e\in\delta^+(X)$, then $\delta^+(X)$ has an intersection with $J$. Otherwise $\delta^+(\bar{Y}\cup\{x,v,u_1\})$ is a dicut in $D_1$. Hence there exist an arc $h\in{\delta^+(\bar{Y}\cup\{x,v,u_1\})\cap{J_1}}$ where $h$ is from $\bar{Y}$ to $\bar{U}\backslash\bar{Y}$. Therefore $h\in{\delta^+(X)\cap{J}}$.
        \item $x,v\notin{X}$: Then $\delta^+(Y)$ is a dicut in $D$ and therefore in $D_2$. Hence there exist an arc $g\in{\delta^+(Y)\cap{J_3}}$ where $g$ is from $Y$ to $U\backslash{Y}$. Therefore $g\in{\delta^+(X)\cap{J}}$. 
        \item $x\in{X}$ and $v\notin{X}$: This case only can happen if there is no weight-0 arc from $v$ to $x$. Then $\delta^+(U\cup{\bar{Y}}\cup{\{v\}})$ is a dicut in $D$ and also in $D_1$. Hence there exist an arc $k\in{\delta^+(U\cup{\bar{Y}}\cup{\{v\}})\cap{J_1}}$ where $k$ is from $\bar{Y}$ to $\bar{U}\backslash\bar{Y}$. Therefore $k\in\delta^+(X)\cap{J}$. 
    \end{itemize}
This shows that $J, J'$ are dijoins of $D$. By this reduction, we may assume that the weight-$0$ vertices of $D$ are not in a $2$-cut-set with mentioned properties.     
\end{cproof}

Recall that a vertex $v$ a near-source if $|N^-(v)|=1$ and $|N^+(v)|\ge 1$. Similarly a vertex $u$ is a near-sink if $|N^+(u)|=1$ and $|N^-(u)|\ge 1$.  

\begin{claim}\label{clm:near-source}
    Let $v$ be a weight-$0$ vertex in $D$. Then we may assume that $v$ is not a near-source or a near-sink, i.e, \ref{A5} holds for $D$. 
\end{claim}

\begin{cproof}
    By symmetry, assume that $v$ is a near-source weight-$0$ vertex. Let $N^-(v)=\{l\}$ and $N^+(v)=\{r_1, r_2, ..., r_t\}$ for some $t\ge 1$. Let $(D',w')$ be obtained by removing $v$ and adding the weight-$0$ arcs $\{(l,r_i)|i\in[t]\}$. Observe that $D'$ is planar. We show that there is a one to one correspondence between dicuts of $D$ and $D'$. Let $\delta^+(X)$ be a dicut of $D$; Then $\delta^+(X\backslash\{v\})$ is a dicut of $D'$. Let $\delta^+(Y)$ be a dicut of $D'$. If $l\notin{Y}$, then $r_1, r_2, ..., r_t\notin{Y}$ and $\delta^+(Y)$ is also a dicut in $D$. Otherwise if $l\in{Y}$, then $\delta^+(Y\cup\{v\})$ is a dicut in $D$. The correspondence shows that $(D',w')$ is a proper $0,1$-weighted digraph and if $J_1,J_2$ are two disjoint dijoins of $D'$ among the weight-$1$ arcs, they are also dijoins of $D$. By this reduction, we may assume that the weight-$0$ vertices of $D$ are not near-source or near-sink. 
\end{cproof}

The goal is to find two disjoint dijoins of $D$ in its weight-$1$ arcs. Suppose $D$ has weight-$0$ vertices. By Claims \ref{clm:acyclic-whole-graph}-\ref{clm:near-source}, we may assume $D$ is super-proper. By \Cref{lem:bridge-lemma}, we can replace these weight-$0$ vertices by a plane subgraph while ensuring that the minimum weight of a dicut remains $2$. Let $(D',w')$ be the obtained proper $0,1$-weighted planar digraph where the weight-$1$ arcs form a connected and spanning component. Then by \Cref{cor:circle-removal}, $D'$ has two disjoint dijoins $J_1, J_2$ contained in its weight-$1$ arcs. By \Cref{lem:bridge-lemma}, $J_1,J_2$ are also dijoins of $D$. 
\qed

\section{Future Research Directions}
In this paper, we characterized when crossing families admit a ($\cap\cup$-closed) cosigning. An interesting future research direction would be to characterize when such cosignings exist for other classes of set families. Also of interest is further applicability of such cosignings. 

We provided polynomial algorithms for finding a ($\cap\cup$-closed) cosigning when the crossing family is given explicitly. In combinatorial optimization, however, many set families have exponential size. We addressed this by providing an oracle polynomial algorithm for cosigning a well-provided crossing family. A very interesting research direction is to develop such an algorithm for $\cap\cup$-closed cosignings, or argue NP-hardness. 

Finally, we built an outer-planar gadget for a digraph whose vertices, comprised of sources and sinks, are placed around a circle, and every cut coming from a certain cosigned crossing family is covered by an incoming arc. We used this to give a graph-theoretic proof of \Cref{thm:main}. Using techniques from submodular flows and polyhedral integrality, the authors have recently shown that this theorem extends to non-planar digraphs as well~\cite{ADN25}. Another research direction is to find a graph-theoretic proof of this.

\vspace{0.5cm}
\textbf{Acknowledgments.} This work was supported in part by EPSRC grant EP/X030989/1.

\textbf{Data Availability Statement.} No data are associated with this article. Data sharing is not applicable to this article.

\textbf{Disclosure of Interests.} The authors have no competing interests to declare that are
relevant to the content of this article.

{\small \bibliographystyle{abbrv}\bibliography{references}}

\end{document}